\documentclass{article}
\usepackage{spconf,amsmath}
\usepackage{amssymb}
\usepackage{amsthm}
\usepackage{algorithm,algorithmicx,algpseudocode}
\usepackage{graphicx,subfigure}
\usepackage{epstopdf}
\usepackage{xcolor}
\usepackage{multirow}
\usepackage[pdftex,colorlinks,linkcolor=blue,anchorcolor=blue,citecolor=blue]{hyperref}
\usepackage{diagbox}
\newtheorem{lemma}{Lemma}
\newtheorem{theorem}{Theorem}
\newtheorem{assumption}{Assumption}
\newtheorem{corollary}{Corollary}
\newtheorem{definition}{Definition}

\newtheorem{example}{Example}

\allowdisplaybreaks[4]

\newcommand{\bs}[1]{\boldsymbol{#1}}
\newcommand{\norm}[1]{\left\|#1\right\|}
\newcommand{\va}{\mathbf{a}}
\newcommand{\vx}{\mathbf{x}}
\newcommand{\vvx}{\tilde{\mathbf{x}}}
\newcommand{\vvz}{\tilde{\mathbf{z}}}
\newcommand{\vvy}{\tilde{\mathbf{y}}}
\newcommand{\vve}{\tilde{\mathbf{e}}}
\newcommand{\vy}{\mathbf{y}}

\newcommand{\vr}{\mathbf{r}}
\newcommand{\ve}{\mathbf{e}}

\newcommand{\vlambda}{\boldsymbol{\lambda}}
\newcommand{\vdelta}{\boldsymbol{\delta}}
\newcommand{\vmu}{\bs{\mu}}
\newcommand{\vphi}{\bs{\phi}}

\newcommand{\ts}{\mathsf{u}}
\newcommand{\mA}{\mathbf{A}}

\newcommand{\mD}{\mathbf{D}}

\newcommand{\mH}{\mathbf{H}}
\newcommand{\mM}{\mathbf{M}}
\newcommand{\mL}{\mathbf{L}}

\newcommand{\mI}{\mathbf{I}}
\newcommand{\mW}{\mathbf{W}}

\newcommand{\RR}{\mathbb{R}}
\newcommand{\tr}{\mathsf{T}}
\DeclareMathOperator*{\diag}{diag}

\DeclareMathOperator{\sign}{sign}

\usepackage{framed}
\colorlet{shadecolor}{blue!20}

\graphicspath{{figures/}} 

\title{Decentralized ADMM with Compressed and Event-Triggered Communication}
\name{Zhen Zhang,  Shaofu Yang, and Wenying Xu\thanks{Corresponding author: Shaofu Yang (sfyang@seu.edu.cn).}}
\address{Southeast University, Nanjing, China}

\begin{document}
	%\ninept
	%
	\maketitle
	\begin{abstract}
%This paper addresses decentralized optimization problems, where a group of agents cooperate to minimize the sum of their private objective functions by peer-to-peer communication. 
This paper focuses on the decentralized optimization problem, where  agents in a network cooperate to minimize the sum of their  local objective  functions by information exchange and local computation. 
Based on  alternating direction method of multipliers (ADMM), we propose CC-DQM, a  communication-efficient decentralized second-order optimization algorithm that combines  compressed communication with  event-triggered communication.  Specifically,
agents are allowed to transmit the compressed message  only when the current primal variables have changed greatly compared to its  last estimate.
		Moreover, to relieve the computation cost, the update of Hessian is scheduled by the trigger condition.	
To   maintain exact linear convergence under compression,  we compress  the difference between the information to be transmitted and  its estimate  by a general contractive compressor.
	 Theoretical analysis shows  that CC-DQM  can still achieve an exact linear convergence, despite   the existence of compression error and  intermittent communication, if the objective functions are strongly convex and smooth. Finally, we validate the performance of CC-DQM by numerical experiments.
	\end{abstract}
	\section{Introduction}
	\label{sec:intro}
%	Over the past years, decentralized optimization  has drawn increasing attention, due to its extensive application in various areas 
%	including large-scale machine learning\cite{TIT-Predd-Kulkarni-etc2009}, multi-robots network\cite{TSP-Cao-Yu-etc2013}, wireless sensor networks\cite{TSP-Khan-Kar-ect2010}, smart grids\cite{TSG-Liu-Shi-etc2018}, etc.
%In recent years, the decentralized optimization problem has received a lot of attention because it appears in various filed
	In recent years, the decentralized optimization problem has attracted  increasing attention due to its  extensive application in 
	  multi-robots network \cite{TSP-Cao-Yu-etc2013}, smart grids \cite{TSG-Liu-Shi-etc2018}, large-scale machine learning \cite{TIT-Predd-Kulkarni-etc2009}, wireless sensor networks \cite{TSP-Khan-Kar-etc2010},  etc.
	A large number of first-order algorithms including \cite{TAC-Nedic-Ozdaglar-2009,SIAM-Shi-Ling-2015,TCNS-Qu-Li-2018} have been proposed for decentralized optimization problems. Compared with the first-order algorithms which just utilize the gradient of the objective function, the second-order algorithm leveraging the extra  Hessian information can accelerate the convergence.
	Recently,  several decentralized  second-order algorithms are proposed, see \cite{TSP-Eisen-Mokhtari-Ribeiro2019,TSP-Mokhtari-Ling-2017,TSIPN-Mokhtari-Shi-2016,TAC-Mansoori-Wei-2020,TSP-Mokhtari-Shi-etc2016}, to name a few.

%	Communication plays a key role in distributed optimization. In aforementioned distributed  second-order algorithms\cite{TSP-Eisen-Mokhtari-Ribeiro2019,TSP-Mokhtari-Ling-2017,TSIPN-Mokhtari-Shi-2016,TSP-Mokhtari-Shi-etc2016}, each agent is required to communicate with its neighbors at each iterations. In practice, the frequent communication can lead to high energy consumption. This is not desirable for some settings where each node is usually equipped with a limited energy resource. Moreover, many second-order algorithms including \cite{TSP-Eisen-Mokhtari-Ribeiro2019,TSP-Mokhtari-Ling-2017,TSIPN-Mokhtari-Shi-2016}  need to estimate the Newton descent direction by several rounds inner-loop where the extra communication are required. Therefore, it is necessary to reduce the communication burden in distributed second-order method.
%	%The total communication cost is determined by the number of bits transmitted per communication and the number of communication rounds.

%In decentralized optimization, efficient communication is very important. However, the above second-order needCommunication plays a key role in decentralized optimization.
Decentralized optimization relies on communication between agents.
In most existing decentralized  second-order algorithms including \cite{TSP-Eisen-Mokhtari-Ribeiro2019,TSP-Mokhtari-Ling-2017,TSIPN-Mokhtari-Shi-2016,TAC-Mansoori-Wei-2020,TSP-Mokhtari-Shi-etc2016},
agents need to transmit accurate updates at every iteration, which can cause high communication costs due to the large payloads and frequent communication. 
High communication costs is undesirable for the  scenarios with limited bandwidth and power constraints.
Moreover, in many second-order algorithms including \cite{TSP-Eisen-Mokhtari-Ribeiro2019,TSP-Mokhtari-Ling-2017,TSIPN-Mokhtari-Shi-2016,TAC-Mansoori-Wei-2020},  to approximate the Newton direction, agents are required to implement several rounds of inner-loop
where extra communication is needed. Hence it is very necessary to improve the communication efficiency of the second-order algorithm.

 To relieve the communication cost,  a popular method is to compress the exchanged message so that fewer bits are transmitted per communication round.
	In decentralized optimization, many algorithms with compressed communication  \cite{ICASSP-Zhu-Hong-etc2016,TWC-Issaid-Chaouki-ect2021,TCOMM-Elgabli-Park-etc2021,TAC-Doan-Maguluri-etc2021,Arxiv-Liu-Li-etc2020,Arxiv-Xiong-Wu-etc2021} have been proposed, among which \cite{ICASSP-Zhu-Hong-etc2016,TWC-Issaid-Chaouki-ect2021,TCOMM-Elgabli-Park-etc2021} belong to
	ADMM-based quantization algorithms where solving subproblem at every iteration is required and
	\cite{TAC-Doan-Maguluri-etc2021,Arxiv-Liu-Li-etc2020,Arxiv-Xiong-Wu-etc2021} belong to first-order communication-compressed methods.
	Based on DGD \cite{TAC-Nedic-Ozdaglar-2009}, the work of \cite{TAC-Doan-Maguluri-etc2021} proposes a well-designed quantization scheme and achieve the exact convergence. 
	In \cite{Arxiv-Xiong-Wu-etc2021}, a gradient tracking algorithm with compressed communication  is introduced, which can converge exactly at a linear convergence. Based on the first-order algorithm NIDS\cite{TSP-Li-Shi-etc2019}, the authors in \cite{Arxiv-Liu-Li-etc2020}  propose a compressed communication
	algorithm which 
	can also achieve linear exact convergence. Despite the progress, few decentralized
	second-order algorithms  with compressed communication are reported.
	
An alternative method  to alleviate the communication cost is intermittent communication which can  reduce total communication rounds.
%The communication-censoring scheme\cite{TAES-Rago-Willett-etc1996} is a very powerful tool to reduce communication round. Its main idea is only to transmit informative message, which is very similar.
The	event-triggered communication scheme is  a  very appealing method in reducing communication rounds. 
It can also be regarded as the celebrated communication-censoring mechanism\cite{TAES-Rago-Willett-etc1996,TSP-Liu-Xu-etc2019}
whose main idea is only to transmit informative message.
%The	event-triggered scheme is  a  very appealing method in distributed control \cite{TSMCS-Xu-Yang-Cao2019},
%as it can reduce the communication frequency, thus reducing the total communication rounds. 
Recently, many decentralized algorithms with event-triggered communication are reported, see \cite{TSP-Liu-Xu-etc2019,Automatica-Chen-Ren-2016,TNNLS-Zhang-Yang-etc2022}, to name a few.
Moreover, there are some  works,  including \cite{TWC-Issaid-Chaouki-ect2021,TNNLS-Liu-Wu-etc2021,TAC-Singh-Data-etc2022}, that combine compressed communication with event-triggered communication.  
%	\cite{TAC-Singh-Data-etc2022}  is a gradient-based first-order method. However, due to the diminishing step size,
%	the convergence rate  is just sublinear.
%	\cite{TWC-Issaid-Chaouki-ect2021,TNNLS-Liu-Wu-etc2021} are ADMM-based methods, which need to solve subproblem at every iteration.
	%By leveraging communication-censored scheme and quantization, the work in \cite{TNNLS-Liu-Wu-etc2021} propose an  ADMM-based “higher-order” method to solve the dynamic decentralized optimization problem. However, the proposed can 
	%not converge exactly. Combing communication-censored scheme and stochastic quantization, the authors in \cite{TWC-Issaid-Chaouki-ect2021} propose a ADMM-based method which can ensure exact linear convergence. But the proposed algorithm can only be applied to bipartite network.

	In this paper, we aim  at developing a decentralized
	communication-efficient  second-order algorithm with a  linear convergence rate to the exact solution.
Since the communication cost is determined by the total communication rounds and the bits per communication round,
	we improve communication efficiency from these two aspects.
	Our main contributions are as follows.
	\begin{itemize}
		\item Based on ADMM,  We develop a communication-efficient second-order algorithm by combining communication compression with event-triggered communication termed  CC-DQM.  Compared with our prior work C-DQM \cite{TNNLS-Zhang-Yang-etc2022}, an event-triggered communication algorithm,
		CC-DQM can save the transmitted bits per communication round.
		Compared with the existing quantized ADMM \cite{ICASSP-Zhu-Hong-etc2016,TNNLS-Liu-Wu-etc2021}, CC-DQM can achieve an exact convergence due to the
		  implementation of a totally different contractive compressor. 
		  Compared with CQ-GGADMM \cite{TWC-Issaid-Chaouki-ect2021}, CC-DQM can be applicable  to a 
		  general  contractive compressor, not just a specific quantizer. Moreover,  CQ-GGADMM can only apply to bipartite graphs while C-DQM can apply to 
		  non-bipartite graphs.
		\item We theoretically prove that CC-DQM can achieve an exact linear convergence if the objective functions are strongly convex and smooth. Numerical experiments demonstrate the effectiveness and efficacy of the proposed algorithm.
	\end{itemize}
\section{Problem Setup}\label{sec:preliminaries}
Consider  $n$ agents connected through a communication network cooperatively solve the following consensus optimization problem 
\begin{equation}
	\min_{\vx\in \RR^d} \sum_{i=1}^{n} f_{i}(\vx),  \label{eq:optimization_problem}
\end{equation}
where $\vx$ refers to the decision variable and   
$f_i: \RR^d \to \RR$ is the local objective function owned  by agent $i$.
Denote the communication graph as
$\mathcal{G} = \{\mathcal{V}, \mathcal{E}\}$, where $\mathcal{V} = \{1,2, \cdots, n\}$ is the set of agents and $\mathcal{E} \subset \mathcal{V} \times \mathcal{V}$ is the set of edges. $(j,i)\in \mathcal{E}$ implies that message can be transmitted from  agent $j$ to agent $i$. Moreover,  there  does not exist self-loop in $\mathcal{G}$, i.e., $(i,i)\notin \mathcal{E}$. 
We use 
 $\mathcal{N}_i = \{j\;|\;(j,i)\in \mathcal{E}\}$ to represent  the set of neighbors of agent $i$ and $d_i=|\mathcal{N}_i|$ to represent the degree of agent $i$. 
%The degree matrix of $\mathcal{G}$ is defined as $\mD =\diag\{d_1, d_2, \cdots, d_n\}$. 
 The degree matrix is represented   by $\mD =\diag\{d_1, d_2, \cdots, d_n\}$ and denote the adjacent matrix of $\mathcal{G}$  as $\mW$, where $w_{ij}=1$ if $(j,i)\in \mathcal{E}$ and $w_{ij}=0$ otherwise. 
 The signed Laplacian matrix is defined as $\mL = \mD - \mW$ and the  unsigned Laplacian matrix 
is defined as $\mL_{\ts} = \mD + \mW$.
%Define $\mL = \mD - \mW$ and $\mL_{\ts} = \mD + \mW$ as Laplacian matrix and unoriented Laplacian matrix, respectively. Clearly, $\mL_{\ts} = 2 \mD - \mL$.
 Denote the eigenvalues of $\mL$ with ascending order as $\lambda_1 \leq \lambda_2 \leq \cdots \leq \lambda_n$. Similarly, 
 we use  $\hat{\lambda}_1 \leq \hat{\lambda}_2 \leq \cdots \leq \hat{\lambda}_n$ to represent the eigenvalues of $\mL_{\ts}$.
The Euclidean norm of vector $\vx$ is denoted by $\|\vx\|$.
\section{Algorithm Development}\label{sec:algorithm}
In this section, we develop a communication-efficient decentralized second-order method.
%In what follows, we introduce the development of this algorithm.
To solve problem \eqref{eq:optimization_problem}, based on ADMM, the authors in \cite{TSP-Mokhtari-Shi-etc2016} proposed an  elegant  second-order method DQM, where the update of agent $i$ is as follows:
%\begin{subequations}\label{eq:DQM-Compact}
%	\begin{align}
%		\vvx_{k+1} ={} & \vvx_{k}-(2c\mD+\mH_k)^{-1}\bigg( \nabla f(\vvx_k) + \vphi_{k} + c \mL \vvx_{k} \bigg)   \label{eq:DQM-Compact-x}\\
%		\vphi_{k+1}={} & \vphi_{k} + c \mL \vvx_{k+1}, \label{eq:DQM-Compact-phi}
%	\end{align}
%\end{subequations}
%where $ \mH_k = \nabla^2 f(\vvx_k)$.
\begin{subequations}
	\begin{align}
		\vx_{i,k+1} = {}& \vx_{i,k}-\Big( 2 c d_i \mI+ \nabla^2 f_i(\vx_{i,k})  \Big)^{-1}\bigg(  \nabla f_{i}(\vx_{i,k})    \notag \\
		& + c \sum_{j \in \mathcal{N}_{i}} (\vx_{i,k} - \vx_{j,k}) + \vphi_{i,k} \bigg)  \\
		\vphi_{i,k+1}={}& \vphi_{i,k} + c \sum_{j \in \mathcal{N}_{i}} (\vx_{i,k+1} - \vx_{j,k+1}),
	\end{align}
\end{subequations}
{where the penalty  parameter $c$ is a positive constant.}
DQM is an ADMM-type algorithm, which reduces the computation burden of DADMM \cite{TSP-Shi-Ling-etc2014} by approximating the objective function  quadratically. 
In DQM, information is required to be transmitted at every iteration, which is undesirable for  settings where the communication source is limited. To relieve the communication burden of DQM, in our prior work \cite{TNNLS-Zhang-Yang-etc2022}, a communication-censored mechanism is leveraged to reduce the communication round. 
%It has been shown in \cite{TNNLS-Zhang-Yang-etc2022} that compared with DQM, C-DQM can reduce communication round greatly, thus causing a sharp decrease of communication cost. 
In order to further reduce communication costs, we not only schedule the communication instants by communication-censored mechanism but also compress the exchanged information. The resulting algorithm is termed  communication-censored and communication-compressed DQM, abbreviated as CC-DQM. 

\textbf{Communication compression}.
The compression scheme we implement is a common $\delta$-contractive compressor, which is defined as follows:
\begin{definition}\label{def-compressor}
The compressor  $\mathcal{C}$ $:$ $\RR^d\rightarrow \RR^d$ is called  $\delta$-contractive compressor if it  satisfies  
	\begin{align}
		\mathbb{E}\big(\|\vx-\mathcal{C}(\vx)\|^2\big) \leq \delta \|\vx\|^2 \quad \forall x \in \RR^d,
	\end{align}   
	where  $0\leq\delta< 1$.
\end{definition}
Many important sparsifiers and quantizers satisfy definition  \ref{def-compressor}.
Next, We  introduce some  contractive compression operators.
\begin{example}\label{example-biased-quantize}
	\cite{NIPS-Sun-Chen-etc2019}
$\mathcal{C}(\vx)=q(\vx)\tau-\|\vx\|_{\infty} \mathbf{1}_d,$
where $[q(\vx)]_i=\lfloor\frac{[\vx]_i+\|\vx\|_{\infty}}{\tau}+\frac{1}{2}\rfloor, \tau=2\|\vx\|_{\infty}/(2^b-1).$	 
\end{example}
%Example \ref{example-biased-quantize} is a deterministic  quantizer and $32+bd$ bits are required  to   quantize a vector with $d$ dimensions. % with $\frac{1}{(2^b-1)^2}\leq \delta\leq \frac{d}{(2^b-1)^2}$
%It is easy to check Example \ref{example-biased-quantize} satisfies Definition \ref{def-compressor} with $\delta\leq \frac{d}{(2^b-1)^2}$.
%In Example 
%\ref{example-biased-quantize}, $32+bd$ bits are required  to   transmitted at every iteration.
%Next, we introduce an unbiased contractive compressor.
\begin{example}\label{example-unbiased-quantize}
\cite{Arxiv-Liu-Li-etc2020}
Denote $\sign{(\vx)}$ and $|\vx|$ as the elementwise sign of $\vx$ and the  elementwise     absolute value of $\vx$, then the compressor is defined as
$$\mathcal{C}(\vx)=\big(\|\vx\|_{\infty}\sign{(\vx)}2^{-(b-1)}\big)\cdot\lfloor\frac{2^{b-1}|\vx|}{\|\vx\|_{\infty}}+\mathbf{u}\rfloor,$$
where $\cdot$ stands for the  the Hadamard product and $\mathbf{u}$ is a random vector uniformly distributed in $[0,1]^d$. 
\end{example}
%In Example \ref{example-unbiased-quantize}, $\delta$ lies in $[\frac{1}{4(2^{b-1})^2},\frac{d}{4(2^{b-1})^2}]$ and $32+(b+1)d$ bits are required to quantize a vector
%with $d$ dimensions. 
Example \ref{example-biased-quantize} is a deterministic  quantizer and $32+bd$ bits are required  to   quantize a vector with $d$ dimensions.
Example \ref{example-unbiased-quantize} is a stochastic quantizer and $32+(b+1)d$ bits are required to quantize a vector
with $d$ dimensions. 

{For any agent $i$, 
since it can not get $\vx_{j,k}$, the exact iterates  of  its neighbors $j$, to estimate $\vx_{j,k}$ and its iterates $\vx_{i,k}$, two state variables
  $\vy_{j,k}$ and $\vy_{i,k}$ are introduced respectively.}
It is worth noting what we compress is $\vx_{i,k+1}-\vy_{i,k}$, the difference between the decision variable $\vx_{i,k+1}$ and the state variable $\vy_{i,k}$.
% kept by
%$\mathcal{N}_i \cup i$.

\textbf{Communication-censored mechanism(Event-triggered communication mechanism).}
The key idea of this mechanism is that  communication
is allowed only when the difference between the current decision variable and the latest estimate is sufficiently large.
Specifically, if the innovation $\|\vx_{i,k+1}-\vy_{i,k}\|$ is greater than the threshold $\mu_k$, agent $i$ will compress $\vx_{i,k+1}-\vy_{i,k}$ and transmit it to the neighbors.   Otherwise, agent $i$ does not transmit any message. After receiving all  the information, agent $i$
updates  the state  variables $\vy_{j,k+1}$ with $j\in i\cup \mathcal{N}_i$.
Moreover, to reduce the computation cost resulting from calculating the inverse of $2cd_i\mI+\nabla f(\vx_{i,k})$, the update of Hessian is scheduled by th triggered condition. 
Specifically, for agent $i$,
if  communication is un-triggered at iteration $k$, then it  does not need to perform the matrix inversion step at iteration $k+1$.
The detailed procedure of our proposed algorithm is shown in Algorithm \ref{algo:PC-DQM}.

\begin{algorithm}[tp!]
	\caption{CC-DQM}\label{algo:PC-DQM}
	\begin{algorithmic}[1]
		\State {For any agent $i$, randomly	choose $\vx_{i,0}\in \RR^d$.  Let $\vphi_{i,0}=\mathbf{0}, \vy_{i,0}=\mathbf{0}$.}
		\For {$k = 0,1,2,\cdots$}
		\For{$i=1$ to $N$}
		\State Update  $\vx_{i,k+1}$ by eq. \eqref{eq:PC-DQM-x};	
		\If{ $\|\vx_{i,k+1}-\vy_{i,k}\|\geq\mu_k$}
		\State Compute $\mathcal{C}(\vx_{i,k+1}-\vy_{i,k})$;
		\State Transmit $\mathcal{C}(\vx_{i,k+1}-\vy_{i,k})$;
		\State Let $\vy_{i,k+1}=\mathcal{C}(\vx_{i,k+1}-\vy_{i,k})+\vy_{i,k}$.
		\Else
		\State Let $\vy_{i,k+1}=\vy_{i,k}$;
		\State Do not send any message.
		\EndIf
		\State Update  $\vphi_{i,k+1}$ by eq. \eqref{eq:PC-DQM-phi}.
		\EndFor
		\EndFor
	\end{algorithmic}
\end{algorithm}
According to the above discussion, we give the update of $\vx_{i}$ and $\vphi_i$, which are as follows:
\begin{subequations}\label{eq:PC-DQM}
	\begin{align}
		\vx_{i,k+1} = {}& \vx_{i,k}-\Big( 2 c d_i \mI+ \nabla^2 f_i(\vy_{i,k})  \Big)^{-1}\bigg(  \nabla f_{i}(\vx_{i,k})    \notag \\
		& + c \sum_{j \in \mathcal{N}_{i}} (\vy_{i,k} - \vy_{j,k}) + \vphi_{i,k} \bigg)  \label{eq:PC-DQM-x}\\
		\vphi_{i,k+1}={}& \vphi_{i,k} + c \sum_{j \in \mathcal{N}_{i}} (\vy_{i,k+1} - \vy_{j,k+1}). \label{eq:PC-DQM-phi}
	\end{align}
\end{subequations}
%\begin{rem}[Comparsion with existing algorithms]
Compared with the quantized
	ADMM \cite{ICASSP-Zhu-Hong-etc2016,TNNLS-Liu-Wu-etc2021}, CC-DQM can converge exactly and enjoy a smaller computation cost since it does not 
	need to solve a subproblem at every iteration.
	Compared with the communication-censored ADMM \cite{TSP-Liu-Xu-etc2019,TNNLS-Zhang-Yang-etc2022,TSIPN-Li-Liu-etc2020}, CC-DQM can reduce  the transmitted bits per communication, thus 
	relieving the communication cost greatly.
	Moreover, as we will show later, compared with the  quantized first-order method \cite{Arxiv-Liu-Li-etc2020,TAC-Singh-Data-etc2022}, CC-DQM enjoys a faster convergence rate, thus achieving a smaller communication cost.
% {\color{red}The work in \cite{Arxiv-Liu-Zhang-etc2022} proposed an elegant second-order algorithm, which  can 
%		achieve an asymptotic local super-linear convergence by $m$ rounds multi-step consensus, where several round extra communication is needed. 
%		Moreover, to track the average of the global Hessian, agents in  \cite{Arxiv-Liu-Zhang-etc2022} need to exchange Hessian in that paper. CC-DQM just needs to transmit the vectors.}
The work in \cite{Arxiv-Liu-Zhang-etc2022} proposed an elegant compressed second-order decentralized algorithm, which can achieve an asymptotic local super-linear convergence. Compared with  CC-DQM, it reduce the communication round by accelerating the convergence rate not by intermittent communication. Moreover, due to the exchange of Hessian and multi-step consensus, it may transmit more  bits per communication round.
%\end{rem}

\section{Convergence Results}\label{sec:convergence-analysis}
In this section, we will show the convergence properties of CC-DQM. 

\begin{assumption}\label{assump:f}
	The local objective function $f_i$ is $v_i$-strongly convex and its gradient is  $\ell_i$-Lipschitz continuous, i.e., $\forall x,x'\in \RR^d$, $\left\langle\nabla f_{i}(x')-\nabla f_{i}(x), x'-x\right\rangle \geq v_i\left\|x'-x\right\|^{2}$,
	and $\left\|\nabla f_i(x')-\nabla f_i(x)\right\| \leq \ell_i \|x'-x\|$.
\end{assumption}

%\begin{assumption}\label{assump:f-hessian}
%	Each $\nabla^2 f_i$ is $\kappa_i$-Lipschitz continuous, i.e., $\forall x, y\in \RR^d$, $\left\|\nabla^2 f_i(x)-\nabla^2 f_i(y)\right\| \leq \kappa_i \|x-y\|$.
%\end{assumption}

%\begin{assumption}\label{assump:f-hessian}
%	The gradient of each $f_i$ is $\ell_i$-Lipschitz continuous. That is, $\forall x,y\in \RR^d$, $\left\|\nabla f_i(x)-\nabla f_i(y)\right\| \leq \ell_i \|x-y\|$.
%\end{assumption}

\begin{assumption}[Communication Graph] \label{assump:graph}
	$\mathcal{G}$ is undirected, connected and $\mL_\ts$ is positive definite.
\end{assumption}
Assumption \ref{assump:f} is very common in decentralized optimization. Under Assumption \ref{assump:f}, we can know $f$ is $v$-strongly convex and $\ell$-smooth, with  ${v} = \min_{i}\{v_i\}$ and ${\ell}=\max_i\{\ell_i\}$.
Assumption \ref{assump:graph} implies that $\mL$ is semi-positive definite with a simple zero eigenvalue.  Note that a positive definite $\mL_{\ts}$ means $\mathcal{G}$ is non-bipartite.  
We first give the convergence result when the event-triggered communication is absent.
\begin{theorem}\label{theorem:concractive-compression}
	Under Assumptions \ref{assump:f} and \ref{assump:graph}, let $\mathcal{C}$ be  $\delta$-contractive compressor,
	 in CC-DQM,
	 if $\mu_k=0$, 
	$c$ and $\delta$ are chosen such that
	\begin{align}  
		\frac{\delta}{(1-\sqrt{\delta})^2}< \frac{G(\beta)}{3c\lambda_n+2c\beta\lambda_n+\frac{c\lambda_n^2}{\beta\lambda_2}},\label{eq:condition-compress}
	\end{align}
	with $ G(\beta)>0$, $\beta>\frac{\ell^2}{2c\lambda_2v}$,
	where 
	\begin{align*}
		G(\beta)={	\frac{c \hat{\lambda}_1}{2}-\frac{2c\beta\lambda_2\ell^2}{2c\beta\lambda_2v-\ell^2}-\frac{ (c^2 \hat{\lambda}_n^2 + 4\ell^2 )}{c\beta\lambda_2 }},
	\end{align*}
	then the sequence {$\mathbb{E}(\vvx_k)$} with	$\vvx_k=[\vx_{1,k},\dots,\vx_{n,k}]$ is convergent to the optimal solution
	 $\vx^*$ at a  linear rate $\mathcal{O}(\sigma^k)$ with $0<\sigma<1$.
\end{theorem}
% The convergence rate is determined by the objective function, the communication graph and the compression parameter $\delta$.
%A larger $\delta$ means a larger error resulted from compression. Therefore, intuitively, a larger $\delta$ causes a slower  convergence rate.
% However, in experiments, we found that the convergence rate of CC-DQM is nearly the same as that of DQM where
%the exchanged information is uncompressed when $\delta$ does not exceed a certain threshold. We may conclude that $\delta$ can not dominate the convergence rate unless $\delta$ is large enough. 
%
%
%
%Moreover, according to our analysis, the convergence rate 
%can not be faster than $\mathcal{O}(\delta^{k/4})$.
The introduction of compression makes the update inexact and therefore may slow down the convergence rate.
But when the $\delta$ is not very large, the effect is almost negligible.
	To satisfy \eqref{eq:condition-compress}, $\delta$ can not be too large, which means that excessive compression of the information to be transmitted should be avoided.
	 The RHS of \eqref{eq:condition-compress}  has a {global  maximum} $F^*$ only determined by the communication graph,  meaning that the choice of $\delta$
 is  related to the  graph  but not  the objective function. 
When  $\delta$ is chosen such that the LHS of \eqref{eq:condition-compress} is less than $F^*$,
	then there always exists a sufficiently large $c$ such that \eqref{eq:condition-compress}  holds.
	Moreover, when we adopt Example \ref{example-biased-quantize} or Example \ref{example-unbiased-quantize},  $\delta$ decays  exponentially as the  number of quantization bits $b$ increases. So a very small $b$ can satisfy the requirement, which will be demonstrated in our experiment.   
The restriction on $\delta$ implies that to ensure a linear convergence, the decaying rate of the compressed error can not be too slow.
%In fact, if we suppose the compressor is unbiased, then 
%the restriction can be removed.
\begin{corollary}\label{pro:unbiased-compressor}
	Under Assumptions \ref{assump:f} and \ref{assump:graph}, let $\mathcal{C}$ be  $\delta$-contractive compressor, when $\mu_k=0$ and $\mathcal{C}$ is unbiased, i.e. $\mathbb{E}\big(\mathcal{C}(\vx)\big)=\vx$, if
	 $\frac{\delta}{(1-\sqrt{\delta})^2}<\frac{\hat{\lambda}_1}{3\lambda_n}$ and    $c>\frac{\ell^2}{v}\frac{2(1-\sqrt{\delta})^2}{\hat{\lambda}_1(1-\sqrt{\delta})^2-3\lambda_n\delta}$,
	 the sequence {$\mathbb{E}(\vvx_k)$} is convergent to the optimal solution
	$\vx^*$ at a linear rate.
\end{corollary}
%	Proposition \ref{pro:unbiased-compressor} shows that when the compressor is unbiased, the restriction of $\delta$ in Theorem \ref{theorem:concractive-compression}
%	can be removed. However, according to our analysis, a deterministic compressor may be preferable, because the stochastic compressor causes a variance of  approximation error, which may slow the convergence rate.
Finally, we will give the result of combining the compression with the communication-censored mechanism.
\begin{theorem}\label{thm:main-result}
	Under Assumptions \ref{assump:f} and \ref{assump:graph}, let $\mathcal{C}$ be  $\delta$-contractive compressor, if $c$ and $\delta$ are chosen such that \eqref{eq:condition-compress} holds and
	$\mu_k=\alpha\rho^{k-1}$ { with $\alpha>0$, $0<\rho<1$,} 
the sequence {$\mathbb{E}(\vvx_k)$} is convergent  to the optimal solution
$\vx^*$ at a linear rate $\mathcal{O}(\tilde{\sigma}^k)$, where $\tilde{\sigma}=\max(\sigma,\rho)$.
\end{theorem}

Theorem \ref{thm:main-result} shows that CC-DQM can still achieve an exact and linear convergence  after combining event-triggered communication with compression if  $\mu_k$  decays linearly. 
It is worth noting that the convergence rate parameter
$\mathcal{\tilde{\sigma}}$ equals  $\max(\sigma,\rho)$, which
implies that
the convergence rate of CC-DQM can not exceed the decaying rate of the threshold. 
%Moreover, % the existence of event-triggered communication does not change the condition  which  $c$ and $\delta$  should satisfy
%Moreover,  the existence of event-triggered communication does not change the exact linear convergence condition  \eqref{eq:condition-compress}.
%does not change 
%in order to achieve an exact and linear convergence.

\section{Numerical Experiments}\label{sec:simulation}
%In this section, we will illustrate the performance of CC-DQM by solving a logistic regression problem. 
%The dataset consists of German credit data from the UCI Machine Learning Repository. 
%Consider a random network composed of $N=100$ agents, each of whom owns $|\mathcal{D}_i|=10$ samples. The objective function is in the following form:
%\begin{equation}\label{lossfun}
%	\min _{\vx \in \mathbb{R}^{d}} f(\vx)=\sum_{i=1}^{N} \frac{1}{\mathcal{D}_i} \sum_{j=1}^{\mathcal{D}_i} \log \left(1+ e^{-b_{ij} \vx^{\tr} \va_{ij}}\right),
%\end{equation}
%where $\va_{ij} \in \RR^{24}$ is the feature vector and $b_{ij}\in\{1,-1\}$ is the label. 
%To evaluate the performance of our algorithm, we use  $\mathrm{err}_k := \frac{\left\|\mathbf{x}_k-\mathbf{x}^*\right\|^2}{\left\|\mathbf{x}_0-\mathbf{x}^*\right\|^2}$ to measure the convergence.
%We tune the parameter $c$ such that DQM can achieve the fastest convergence and $\rho$ is tuned to make C-DQM get the best communication rounds performance.
%In this section
This section provides numerical simulations  to show the performance of CC-DQM.
We consider a  logistic regression problem where the dataset comprises  German credit data from the UCI Machine Learning Repository.
Define the connectivity ratio  $\tau$ as the number of  edges 
divided by $\frac{n(n-1)}{2}$.
The communication graph is a stochastic graph with connectivity ratio $\tau=0.4$.
There exist $n=100$ agents in the graph and each agent holds $m_i=10$ samples. 
The optimization problem is $\min _{\vx \in \mathbb{R}^{d}} f(\vx)=\sum_{i=1}^{n} \frac{1}{m_i} \sum_{j=1}^{m_i} \log \left(1+ e^{-b_{ij} \vx^{\tr} \va_{ij}}\right),$ % as follows:
%The optimization problem is as follows:
%\begin{equation}\label{lossfun}
%	\min _{\vx \in \mathbb{R}^{d}} f(\vx)=\sum_{i=1}^{n} \frac{1}{m_i} \sum_{j=1}^{m_i} \log \left(1+ e^{-b_{ij} \vx^{\tr} \va_{ij}}\right),
%\end{equation}
where $\va_{ij} \in \RR^{24}$ represents the feature vector  and $b_{ij}\in\{1,-1\}$ represents the label. 
Moreover, we define $\mathrm{err}_k := \frac{\left\|\mathbf{x}_k-\mathbf{x}^*\right\|^2}{\left\|\mathbf{x}_0-\mathbf{x}^*\right\|^2}$ to measure the convergence.
We tune the parameter $c$ such that DQM can achieve the fastest convergence and $\rho$ is tuned to make C-DQM get the best communication rounds performance.

We first compare CC-DQM with the existing 
 ADMM-type communication-efficient algorithm including COCA, DQM, C-DQM. 
COCA \cite{TSP-Liu-Xu-etc2019} is a communication-censored ADMM, where agents need to solve subproblems at every iteration.
C-DQM \cite{TNNLS-Zhang-Yang-etc2022} is the communication-censored version of DQM.
In this experiment, the compressor we implement is Example \ref{example-biased-quantize}, the deterministic quantizer.
The relevant results can be seen in Fig. \ref{fig:com-ADMM-type}.
As we can see,  CC-DQM is the most communication efficient as it can not only save communication rounds but also reduce the transmitted bits per communication round.
Our Theorem \ref{theorem:concractive-compression} shows that  to achieve a linear and exact convergence,
the number of quantization bits can not be too small.
In our experiment, CC-DQM can converge linearly when we implement 
2 bit deterministic quantization.
Moreover, it is worth noting that the convergence of CC-DQM is  nearly the same as DQM.
	\begin{figure}[htb]
	
	\begin{minipage}[b]{.48\linewidth}
		\centering
		\centerline{\includegraphics[width=4.0cm]{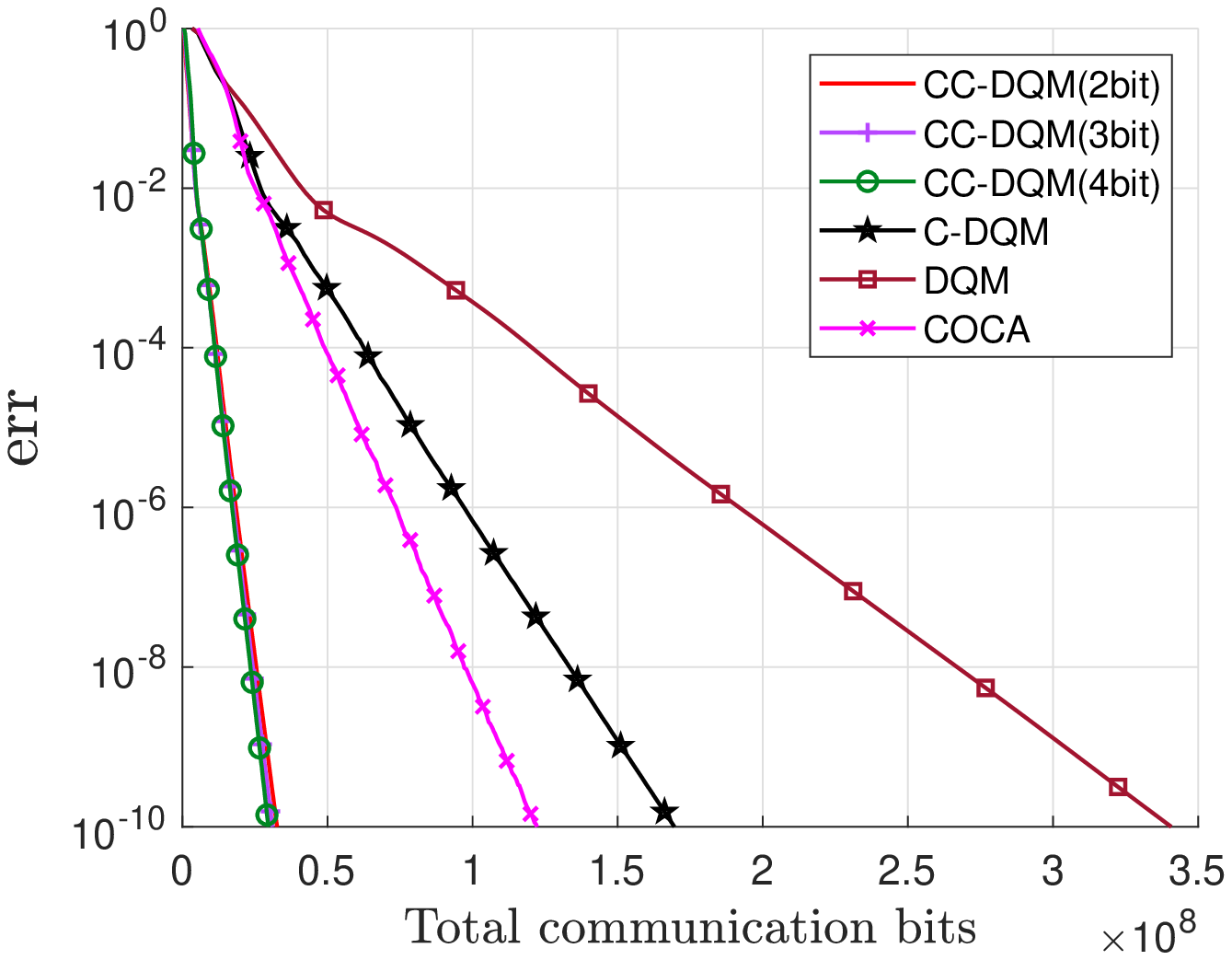}}
		%  \vspace{2.0cm}
	%	\centerline{(a) Result 1}\medskip
	\end{minipage}
	\begin{minipage}[b]{.48\linewidth}
		\centering
		\centerline{\includegraphics[width=4.0cm]{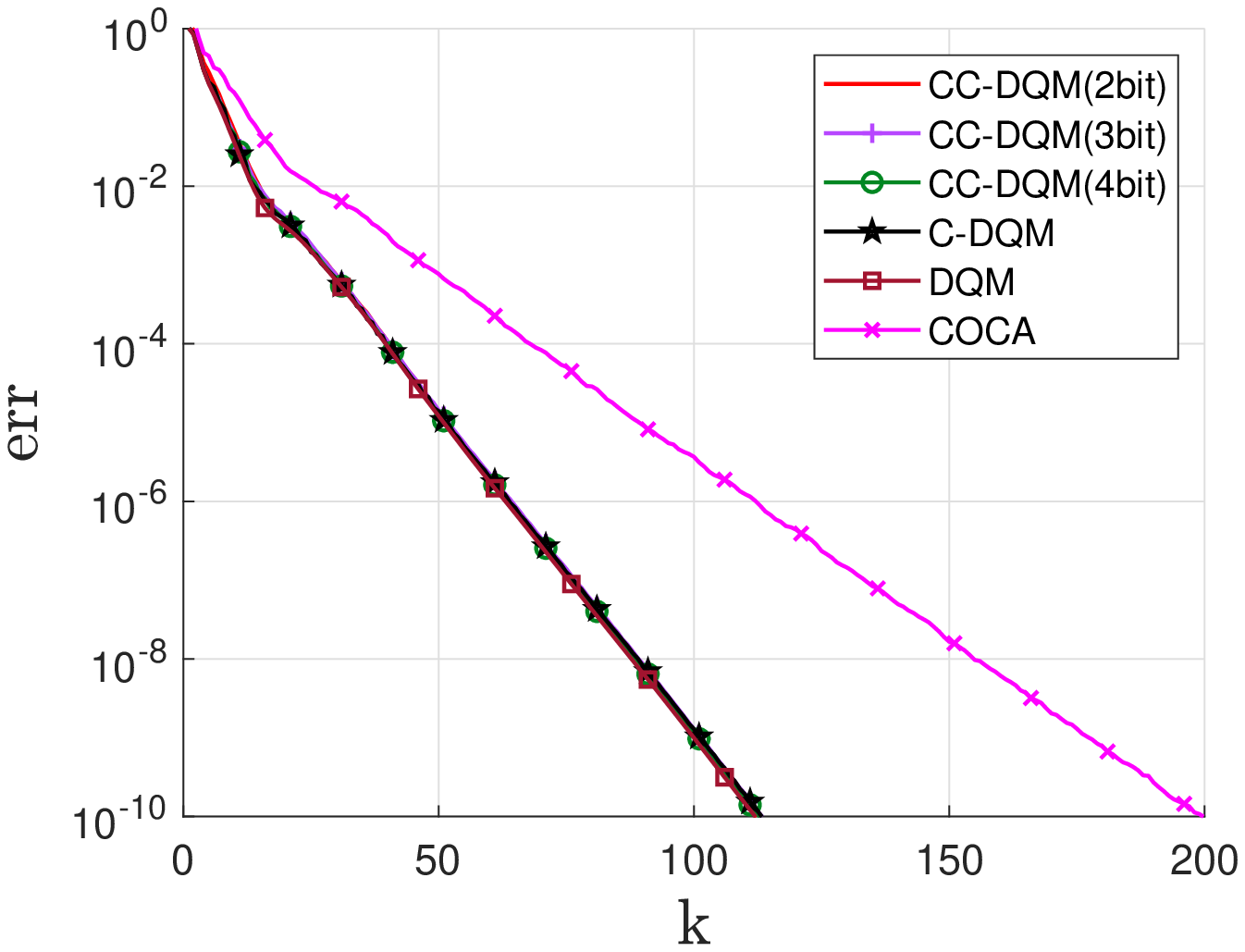}}
		%  \vspace{1.5cm}
	%	\centerline{(b) Results 3}\medskip
	\end{minipage}
	
	\caption{Comparison with the existing ADMM-type communication-efficient algorithm.}
	\label{fig:com-ADMM-type}
\end{figure}
We then compare CC-DQM with the existing first-order communication-efficient  methods, including SPARQ-SGD \cite{TAC-Singh-Data-etc2022} and LEAD \cite{Arxiv-Liu-Li-etc2020}. SPARQ-SGD combines  event-triggered communication with compressed communication and LEAD is a communication-compressed algorithm. 
For a fair  comparison, in SPARQ-SGD,  we use the full gradient  instead of the stochastic gradient.
We consider  biased compressor Example \ref{example-biased-quantize}  and  the unbiased compressor Example \ref{example-unbiased-quantize}. In both  schemes, we let $b=2$. 
As shown in Fig. \ref{fig:bias-First-type}, compared with the existing first-order methods,  no matter what kind of compressor is implemented, CC-DQM always enjoys the smallest 
communication cost. This is because CC-DQM enjoys a faster convergence rate than other algorithms.
%	\begin{figure}[htb]
%	
%	\begin{minipage}[b]{.48\linewidth}
%		\centering
%		\centerline{\includegraphics[width=4.4cm]{Unbias-First-Com}}
%		%  \vspace{2.0cm}
%%		\centerline{Under unbiased compressor}\medskip
%	\end{minipage}
%	%
%	\begin{minipage}[b]{.45\linewidth}
%		\centering
%		\centerline{\includegraphics[width=4.4cm]{Unbias-First-Rate}}
%		%  \vspace{1.5cm}
%%		\centerline{Under biased compressor}\medskip
%	\end{minipage}			
%	
%	\caption{Comparison with the existing first-order communication-efficient algorithms under unbiased compressor.}
%	\label{fig:Unbias-First-type}
%	%
%\end{figure}
	\begin{figure}[htb]
	\vspace{-0.29cm}
	\begin{minipage}[b]{.48\linewidth}
		\centering
%		\centerline{\includegraphics[width=4.4cm]{bias-First-Com}}
\centerline{\includegraphics[width=4.4cm]{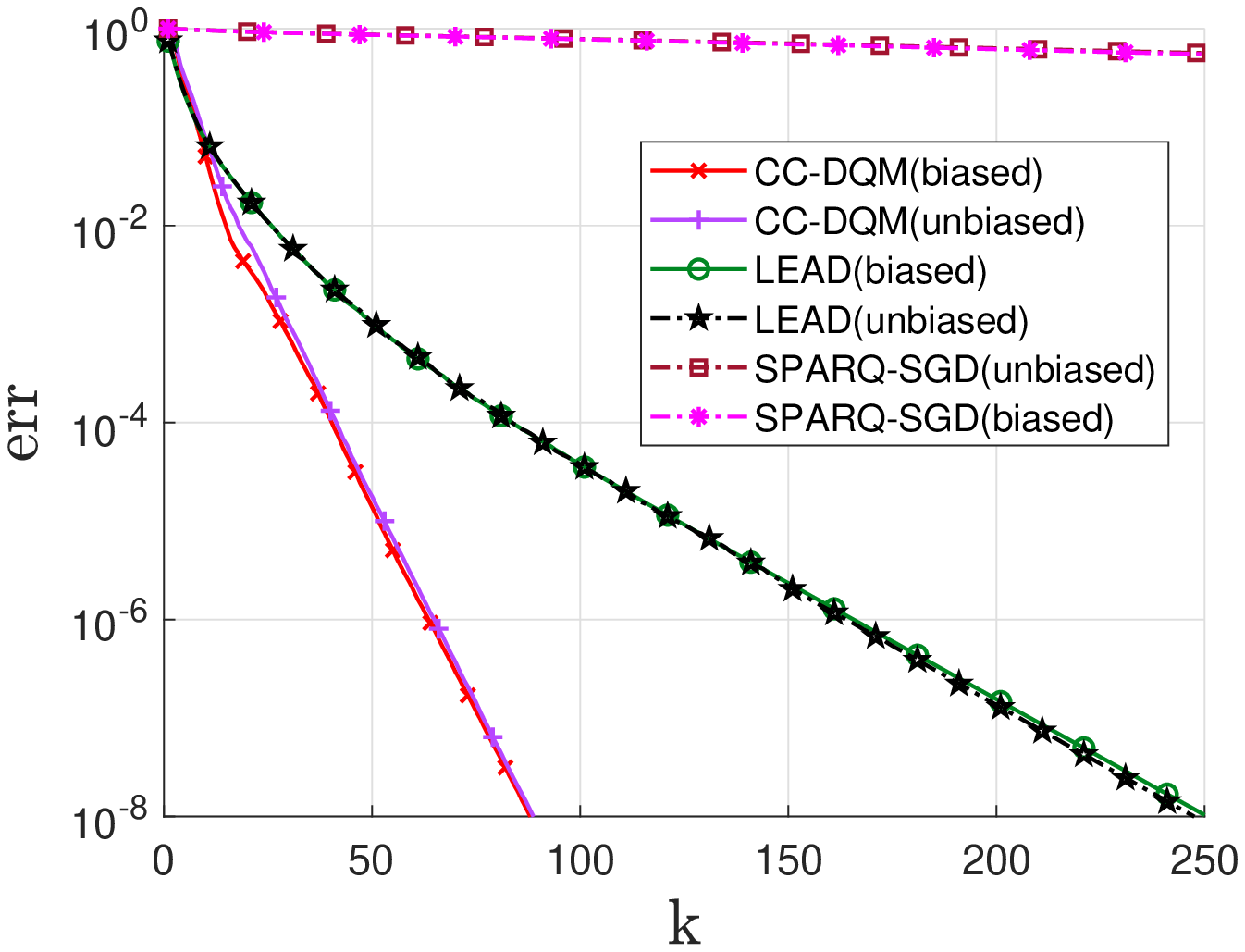}}
		%  \vspace{2.0cm}
		%		\centerline{Under unbiased compressor}\medskip
	\end{minipage}
	\begin{minipage}[b]{.45\linewidth}
		\centering
%		\centerline{\includegraphics[width=4.4cm]{bias-First-Rate}}
\centerline{\includegraphics[width=4.4cm]{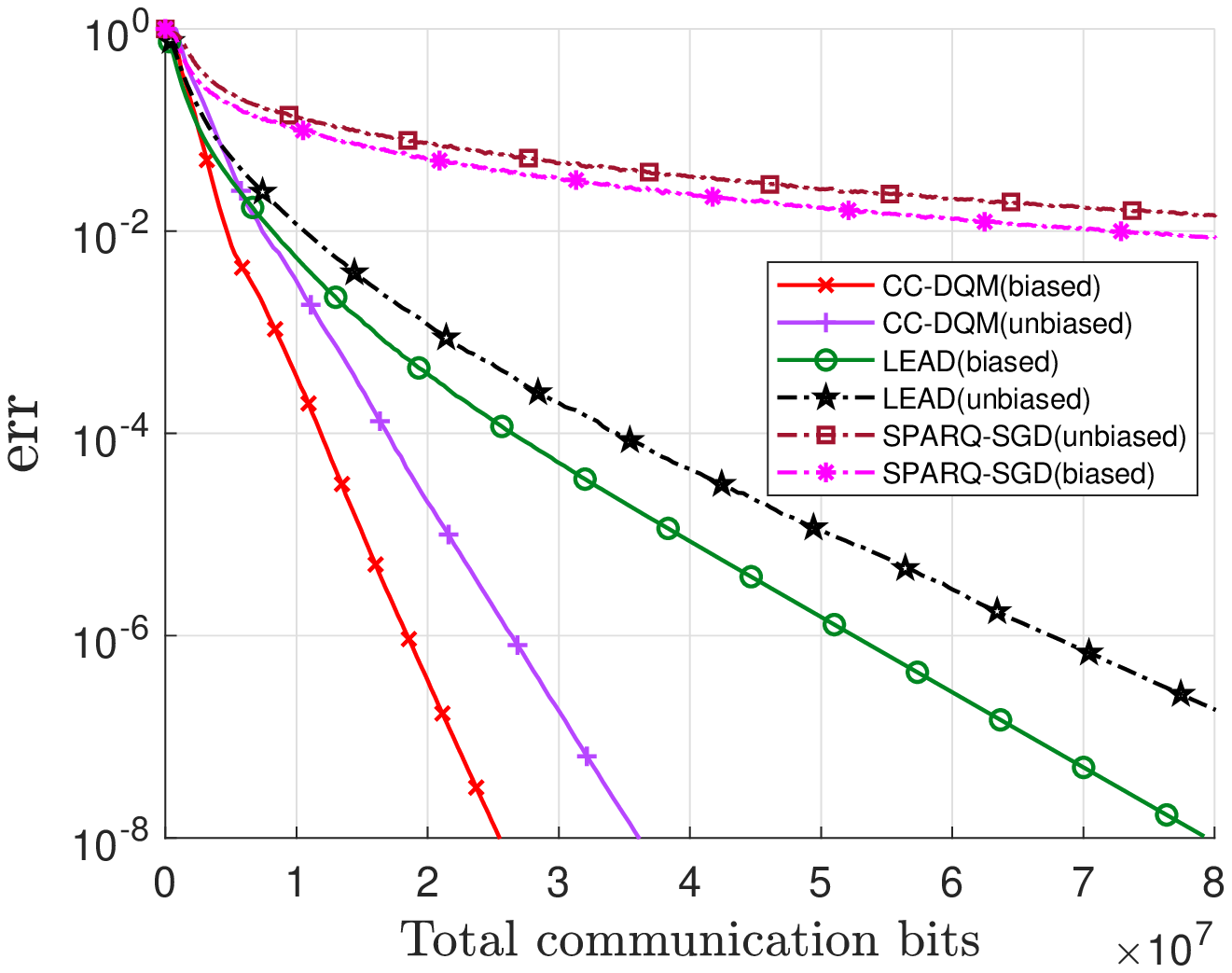}}
		%  \vspace{1.5cm}
		%		\centerline{Under biased compressor}\medskip
	\end{minipage}			
	
	\caption{Comparison with the performance of existing first-order algorithms.}
	\label{fig:bias-First-type}
\end{figure}

\small
	\bibliographystyle{IEEE.bst}
	\bibliography{reference}

\newpage

\appendix

\centerline{\bf \Large Supplementary Material}

\section{Preparation for the Proof}
We first give the matrix form of CC-DQM, which is as follows:
\begin{subequations}\label{eq:PC-DQM-Compact}
	\begin{align}
		\vvx_{k+1} ={} & \vvx_{k}-\tilde{\mD}^{-1}\bigg( \nabla f(\vvx_k) + \vphi_{k} + c \mL \vvy_{k} \bigg)   \label{eq:PC-DQM-Compact-x}\\
		\vphi_{k+1}={} & \vphi_{k} + c \mL \vvy_{k+1} \label{eq:PC-DQM-Compact-phi}
	\end{align}
\end{subequations}
where $\tilde{\mD} = 2c \mD+ \nabla^2 f(\vvy_k)$.
To proof Theorem 1 and Theorem 2, we  introduce a key Lemma estimating  the error caused by
event-triggered communication and compression.
Define $\vve_{k}=\vvy_{k}-\vvx_{k}$.
\begin{lemma} \label{lemma:error-eve-comp}
	Denote $\mathcal{C}$ as the contractive operator with the parameter $\delta \in[0,1),$ in CC-DQM, the error $\vve_{k+1}$ satisfies
	\begin{align}
		\mathbb{E}(\|\vve_{k+1}\|^2)\leq & \sqrt{\delta}	\mathbb{E}(\|\vve_{k}\|^2)+	\frac{\delta}{1-\sqrt{\delta}}	\mathbb{E}(\|\vvx_{k+1}-\vvx_{k}\|^2)  \notag \\
		&+n\mu_{k+1}^2.   \label{eq:error-eve-comp}
	\end{align}
\end{lemma}
\begin{proof}
	For agent $i$ at iteration $k+1$, if $h_{i,k}=\|\vx_{i,k+1}-\vy_{i,k}\|-\mu_{k+1}<0$, then communication is not triggered and $\vy_{i,k+1}=\vy_{i,k}$. So we can get 
	\begin{align}\label{eq:censor-noise-estimate}
		\ve_{i,k+1}=\|\vx_{i,k+1}-\vy_{i,k+1}\|<\mu_{k+1}
	\end{align}
	When $h_{i,k}>0$, $\vy_{i,k+1}=\mathcal{C}(\vx_{i,k+1}-\vy_{i,k})+\vy_{i,k}$.
	So we can know 
	\begin{align}
		\ve_{i,k+1}=&\mathcal{C}(\vx_{i,k+1}-\vy_{i,k})-(\vx_{i,k+1}-\vy_{i,k}) 
	\end{align}	 
	According to the property of compressor, we can obtain:
	\begin{align}
		\mathbb{E}(\| \ve_{i,k+1}\|^2|\vy_{i,k},\vx_{i,k+1}) \leq{}& \delta\|\vx_{i,k+1}-\vy_{i,k}\|^2\notag \\
		\leq {}&  \delta\|\vx_{i,k+1}-\vx_{i,k}-(\vy_{i,k}-\vx_{i,k})\|^2 \notag \\
		\leq {}&    \delta(1+t^{-1})(\|\vx_{i,k+1}-\vx_{i,k}\|^2) \notag \\ 
		& + \delta(1+t)\|\ve_{i,k}\|^2
	\end{align}
	Let $t=\frac{1}{\sqrt{\delta}}-1$ and take expectations, then we can obtain
	\begin{align} \label{eq:quantize-noise-estimate}
		\mathbb{E}(\|\ve_{i,k+1}\|_2^2) \leq \sqrt{\delta}\mathbb{E}(\|\ve_{i,k}\|_2^2)+\frac{\delta}{1-\sqrt{\delta}}\mathbb{E}(\|\vx_{i,k+1}-\vx_{i,k}\|^2).
	\end{align}
	By combing  \eqref{eq:censor-noise-estimate} and \eqref{eq:quantize-noise-estimate}, then we can finish the proof.
\end{proof}

Since CC-DQM is an ADMM-type algorithm, according to \cite{TSP-Shi-Ling-etc2014}, we give 
its optimal condition.
\begin{lemma}\label{lemma:optimal-condition}
	Suppose $(\vvx^*,\vvz^*,\vlambda^*)$ is a primal-dual optimal pair of the augmented Larriangian. Then, it holds that $\mM \vvx^{*} = \mathbf{0}$, $\frac{1}{2} \mM_{\ts} \vvx^{*} = \vvz^{*}$, and there exist a unique $\vmu^*$  lying in the column space of $\mM$
	satisfying $\vphi^* = \mM^{\tr} \vmu^*$, $\vlambda^{*} = [\vmu^*; -\vmu^*]$, and 
	\begin{align}
		\nabla f(\vvx^{*}) + \vphi^{*} & = \mathbf{0}. \label{eq:opt-condition-phi*}
	\end{align}
\end{lemma}
Next, we show the relationship between $\vr_k$ and $\vphi_k$. As $\vphi_0 = \bs{0}$, by recursive computation based on \eqref{eq:PC-DQM-Compact-phi}, we have
$\vphi_{k+1} = \vphi_0 + c \sum_{s=1}^{k+1} \mL \vvy_s = c \sum_{s=1}^{k+1} \mL \vvy_s$. 
As $\mL = \frac{1}{2} \mM^{\tr} \mM$, we further have
\begin{align}
	\vphi_{k+1} ={}& 2 c \mM^{\tr} \vr_{k+1},    \label{eq:phi-r}\\
	\vr_{k+1} ={}& \vr_k + \frac{1}{4} \mM \vvy_{k+1}. \label{eq:rk-rk-1}
\end{align}
Recalling Lemma \ref{lemma:optimal-condition}, by letting $\vr^* = \frac{1}{2c}\vmu^*$, we have $\vphi^* = 2c\mM^{\tr} \vr^*$. 
The remain task is to show the convergence of $(\vvx_k, \vr_k)$ to $(\vvx^*,\vr^*)$.

Define
\begin{align}
	V_k = \frac{c}{2}\mathbb{E}(\norm{\vvx_k - \vvx^*}^2_{\mL_{\ts}}) + 4c\mathbb{E}(\norm{\vr_k - \vr^*}^2)+r\mathbb{E}(\|\vve_k\|^2),    \label{eq:Vk-def}
\end{align}
where  $r$ is a positive constant, which will be determined later.
It is clear that the convergence of CC-DQM is equivalent to $V_k \to 0$ as $k\to \infty$.
Regarding the evolution of $V_k$. We have the following lemma, which is infrastructural for our main result.
\begin{lemma}\label{lemma:qk-estimate}
	Under Assumptions 1, 2, 3 if 
	$c$ and $\delta$ is chosen such that
	\begin{align} 
		\frac{\delta}{(1-\sqrt{\delta})^2}< \frac{G(\beta)}{3c\lambda_n+2c\beta\lambda_n+\frac{c\lambda_n^2}{\beta\lambda_2}},\label{eq:quantize-condition}
	\end{align}
	with $ G(\beta)>0$, $\beta>\frac{\ell^2}{2c\lambda_2v}$,
	where 
	\begin{align*}
		G(\beta)={	\frac{c \hat{\lambda}_1}{2}-\frac{4\ell^2c\beta\lambda_2}{4c\beta\lambda_2v-2\ell^2}-\frac{ (c^2 \hat{\lambda}_n^2 + 4\ell^2)}{c\beta\lambda_2 }},
	\end{align*}
	then there exists $r>0$, $\eta >\frac{c\lambda_2}{2c\lambda_2v-\beta^{-1}\ell^2}$ and $\hat{\sigma}>0$ such that the sequence $V_k$ generated by CC-DQM satisfies
	\begin{align*}
		V_{k+1} \leq \frac{1}{1+\hat{\sigma}} V_k + n\psi\mu_{k+1}^2,
	\end{align*}
	where $\psi = r+\frac{1}{1+\hat{\sigma}}\left(1.5c\lambda_n + 2c{\beta}\lambda_n + \frac{(c\beta^{-1}+4\hat{\sigma} c)\lambda_n^2}{2\lambda_2}\right)$,
	%		\begin{align}\label{eq:delta-k}
	%			&\hspace*{-5mm}\hat{\sigma}=\min  \left.\bigg(\frac{1-\sqrt{\delta}}{\sqrt{\delta}+\frac{2c\lambda_n^2(1+\sqrt{\delta})}{r\lambda_2}}-\frac{\lambda_2(\tilde{\Xi}_1+\tilde{\Xi}_2-(1-\sqrt{\delta})\tilde{\Xi}_1)}{\sqrt{\delta}r\lambda_2+2c\lambda_n^2(1+\sqrt{\delta})},\right. \notag \\
	%			&\left.\frac{c\lambda_2(c\hat{\lambda}_1-2r\frac{\delta}{1-\sqrt{\delta}}-2\tilde{a}_1\frac{\delta}{1-\sqrt{\delta}}-4\eta\ell^2)-2(c^2\hat{\lambda}_n^2+4\ell^2)\beta^{-1}}{8c^2\hat{\lambda}_n^2+32\ell^2+4c^2\lambda_n^2+4c\lambda_2 r\frac{\delta}{1-\sqrt{\delta}}},\right.\notag \\
	%			&\left.\frac{c\lambda_2(2v-\eta) - \ell^2 \beta^{-1}}{c^2 \lambda_2 \hat{\lambda}_n + \ell^2}\right.\bigg)>0,
	%		\end{align}
	
	\begin{align}
		&\hat{\sigma}=\min\bigg\{\frac{1-\sqrt{\delta}}{\sqrt{\delta}+\frac{2c\lambda_n^2(1+\sqrt{\delta})}{r\lambda_2}}-\frac{\lambda_2(\tilde{\Xi}_1+\tilde{\Xi}_2-(1-\sqrt{\delta})\tilde{\Xi}_1)}{\sqrt{\delta}r\lambda_2+2c\lambda_n^2(1+\sqrt{\delta})}, \notag \\
		&\frac{c\lambda_2(c\hat{\lambda}_1-2r\frac{\delta}{1-\sqrt{\delta}}-2\tilde{\Xi}_1\frac{\delta}{1-\sqrt{\delta}}-4\eta\ell^2)-2(c^2\hat{\lambda}_n^2+4\ell^2)\beta^{-1}}{8c^2\hat{\lambda}_n^2+32\ell^2+4c^2\lambda_n^2+4c\lambda_2 r\frac{\delta}{1-\sqrt{\delta}}},\notag \\
		&\frac{c\lambda_2(2v-\eta) - \ell^2 \beta^{-1}}{c^2 \lambda_2 \hat{\lambda}_n + \ell^2}\bigg\}>0, \label{eq:delta-k}  \\
		&\frac{\tilde{\Xi}_2+\tilde{\Xi}_1\sqrt{\delta}}{1-\sqrt{\delta}}<r<\frac{c\hat{\lambda}_1-4\eta\ell^2}{2\frac{\delta}{1-\sqrt{\delta}}}-\frac{c^2\hat{\lambda}_n^2+4\ell^2}{c\lambda_2\frac{\delta}{1-\sqrt{\delta}}\beta}-\tilde{\Xi}_1,   \notag  \\
		&\tilde{\Xi}_1=\frac{3c\lambda_n}{2}+2c\beta\lambda_n+\frac{c\lambda_n^2}{2\beta\lambda_2},\qquad  \tilde{\Xi}_2=\frac{3c\lambda_n}{2}+\frac{c\lambda_n^2}{2\beta\lambda_2}. \notag 
	\end{align}
\end{lemma}

\begin{proof}
	Before proceeding, inspired by \cite{TSP-Mokhtari-Shi-etc2016}, we 
	define the approximated error  on  $\nabla f(\vvx_{k+1})$ as
	$\vdelta_k = \nabla f(\vvx_k) + \nabla^2 f(\vvy_k)(\vvx_{k+1} - \vvx_k) - \nabla f(\vvx_{k+1})$.
	Sine $f$ satisfies $\ell$ smooth, we
	can know  $\|\vdelta_k\|\leq \gamma\|\vvx_{k+1}-\vvx_{k}\|$ with $\gamma=2\ell$.
	Next, we begin to give our proof.
	Denote $\nabla^2 f(\vvy_k)$ as $\tilde{\mH}_k$.
	According to  \eqref{eq:PC-DQM-Compact-x}, we can obtain that
	\begin{align}
	{}&\hspace*{-2mm}	\nabla f(\vvx_{k+1}) \notag  \\
		={} & \nabla f(\vvx_k) + \tilde{\mH}_k(\vvx_{k+1} - \vvx_k) - \vdelta_k  \notag \\
		={} & \tilde{\mD}(\vvx_k -\vvx_{k+1}) - \vphi_k - c \mL \vvy_k + \tilde{\mH}_k(\vvx_{k+1} - \vvx_k) - \vdelta_k  \notag \\
		={} & 2 c \mD (\vvx_k -\vvx_{k+1}) - \vphi_k - c \mL (\vvx_k + \vve_k) - \vdelta_k \notag \\
		={} & c \mL_{\ts} (\vvx_k - \vvx_{k+1}) - c \mL \vvx_{k+1} - \vphi_k - c \mL \vve_k - \vdelta_k  \notag \\
		={} & c \mL_{\ts} (\vvx_k - \vvx_{k+1}) - c \mL (\vvy_{k+1} - \vve_{k+1}) - \vphi_k   
		 - c \mL \vve_k - \vdelta_k \notag \\
		={} & c \mL_{\ts} (\vvx_k - \vvx_{k+1}) - \vphi_{k+1} + c \mL (\vve_{k+1} - \vve_k) - \vdelta_k,  \label{eq:grad-f}
	\end{align}
	where the third equality utilizes $\tilde{\mD} - \tilde{\mH}_k = 2c \mD$ and $\vvy_k = \vvx_k + \vve_k$, the fourth equality utilizes $2\mD = \mL_{\ts} + \mL$.  By noting that $\nabla f(\vvx^*) = -\vphi^*$, we have
	\begin{align}
		&\hspace*{-2mm}   \mathbb{E}\big((\vvx_{k+1}-\vvx^{\ast})^{\tr}(\nabla f(\vvx_{k+1}) - \nabla f(\vvx^{\ast}))\big)   \notag  \\
		={} & \underbrace{c  \mathbb{E}\big((\vvx_{k+1}-\vvx^{\ast})^{\tr} \mL_{\ts} (\vvx_k - \vvx_{k+1})}_{\Xi_1} \big)    \notag  \\
		{}&	+ \underbrace{ \mathbb{E}\big((\vvx_{k+1}-\vvx^{\ast})^{\tr} ( \vphi^* - \vphi_{k+1})\big) }_{\Xi_2}+ \underbrace{\mathbb{E}\big( (\vvx^{\ast} - \vvx_{k+1})^{\tr} \vdelta_k\big)}_{\Xi_4}    \notag \\
		& + \underbrace{ c \mathbb{E}\big((\vvx_{k+1}-\vvx^{\ast})^{\tr} \mL (\vve_{k+1}-\vve_{k})\big)}_{\Xi_3}.  \label{eq:gradf-x}
	\end{align}
	Next, we further estimate the above four terms. Regarding $\Xi_1$, by using $2 \vx^{\tr} \mA \vy = \|\vx+\vy\|_{\mA}^2 -\|\vx\|_{\mA}^2-\|\vy\|_{\mA}^2$, we have
	\begin{align*}
		&\hspace*{-2mm}	\Xi_1 = \\
		{}& \frac{c}{2} \bigg(\mathbb{E}(\norm{\vvx_k - \vvx^*}_{\mL_{\ts}}^2) -\mathbb{E}(\norm{\vvx_{k+1} - \vvx^*}_{\mL_{\ts}}^2)- \mathbb{E}(\norm{\vvx_k - \vvx_{k+1}}_{\mL_{\ts}}^2) \bigg).
	\end{align*}
	Regarding $\Xi_2$, as $\vphi_{k+1} - \vphi^* = 2c \mM^{\tr}(\vr_{k+1} - \vr^*)$, $\mM \vvx^* = \bs{0}$, and $\vr_{k+1} - \vr_k = \frac{1}{4}\mM \vvy_{k+1}$, one has
	\begin{align*}
		\Xi_2 ={}&  - 2 c\mathbb{E}\big( \vvx_{k+1}^{\tr} \mM^{\tr} (\vr_{k+1} - \vr^*)\big) \notag \\
		={}&- 2 c\mathbb{E} \big((\vvy_{k+1} - \vve_{k+1})^{\tr} \mM^{\tr} (\vr_{k+1} - \vr^*)\big)  \notag \\
		={}&  4 c \mathbb{E}\bigg(\norm{\vr_k - \vr^*}^2 - \norm{\vr_{k+1} - \vr^*}^2 - \norm{\vr_k - \vr_{k+1}}^2 \bigg)  \notag \\
		{}& + 2 c \mathbb{E}\big(\vve_{k+1}^{\tr} \mM^{\tr}(\vr_{k+1} - \vr^*)\big) \notag \\
		\leq {}& 4 c\mathbb{E} \bigg(\norm{\vr_k - \vr^*}^2 - \norm{\vr_{k+1} - \vr^*}^2 - \norm{\vr_k - \vr_{k+1}}^2 \bigg)  \notag \\
		{}&	+ 2 c \beta \mathbb{E}(\vve_{k+1}^{\tr} \mL \vve_{k+1}) + c \beta^{-1}\mathbb{E} (\norm{\vr_{k+1} - \vr^*}^2). 
	\end{align*}
	Regarding $\Xi_3$, we have
	\begin{align*}
		\Xi_3 = {}& \frac{c}{2} \mathbb{E}\big(\vvx_{k+1}^{\tr} \mM^{\tr} \mM (\vve_{k+1}-\vve_{k})\big)   \notag \\
		={}&\frac{c}{2} \mathbb{E}\big( (\vvy_{k+1} - \vve_{k+1})^{\tr} \mM^{\tr} \mM (\vve_{k+1}-\vve_{k})\big)   \notag \\
		={}& 2 c\mathbb{E} \big((\vr_{k+1} - \vr_k)^{\tr} \mM (\vve_{k+1} - \vve_k)\big) + c \mathbb{E}\big(\vve_{k+1}^{\tr} \mL (\vve_k - \vve_{k+1} )\big) \notag \\
		\leq {}& \frac{c}{2}\mathbb{E}\big((\vve_{k+1} - \vve_k)^{\tr} \mL (\vve_{k+1} - \vve_k)\big) + 4 c\mathbb{E}\big( \norm{\vr_{k+1} - \vr_k}^2\big) \notag \\
		&+ c \mathbb{E}\big(\vve^{\tr}_{k+1} \mL \vve_k\big).
	\end{align*}
	Regarding $\Xi_4$, we have
	\begin{align*}
		\Xi_4 &\leq {}   \frac{\eta}{2}\mathbb{E}( \norm{\vdelta_k}^2) + \frac{1}{2 \eta} \mathbb{E}(\norm{\vvx_{k+1} - \vvx^*}^2) \notag \\  
		&\leq {} \frac{\eta \gamma^2}{2}\mathbb{E}(\norm{ \vvx_{k+1} - \vvx_k }^2) +\frac{1}{2 \eta}\mathbb{E}(\norm{\vvx_{k+1} - \vvx^*}^2).
	\end{align*}
	Define 
	$$\tilde{V}_{k}=\frac{c}{2}\mathbb{E}(\norm{\vvx_k - \vvx^*}^2_{\mL_{\ts}}) + 4c\mathbb{E}(\norm{\vr_k - \vr^*}^2).$$
	Due to the $v$-strongly convexity of $f$, we have
	$(\vvx_{k+1}-\vvx^{\ast})^{\tr}(\nabla f(\vvx_{k+1}) - \nabla f(\vvx^{\ast})) \ge v \norm{\vvx_{k+1} - \vvx^*}^2$, which yields
	\begin{align}
		&\hspace*{-3mm}		v\mathbb{E}(\norm{\vvx_{k+1} - \vvx^*}^2)  \notag  \\
		\leq {}& \Xi_1 + \Xi_2 + \Xi_3 + \Xi_4    \notag  \\
		\leq {}& \tilde{V}_k - \tilde{V}_{k+1} +  \bigg(\frac{\eta \gamma^2}{2} - \frac{c \hat{\lambda}_1 }{2} \bigg)\mathbb{E}(\norm{\vvx_{k+1} - \vvx_k }^2)  
		\notag \\
		& +  c \beta^{-1}\mathbb{E}(\norm{\vr_{k+1} - \vr^*}^2)  + {c \lambda_n (2\beta + 1)}\mathbb{E}(\norm{\vve_{k+1}}^2) \notag \\
		&	+ { c \lambda_n}\mathbb{E}(\norm{\vve_k}^2) 
		+ c \lambda_n\mathbb{E}(\|\vve_{k}\|\|\vve_{k+1}\|)+ \frac{1}{2 \eta}\mathbb{E}(\norm{\vvx_{k+1} - \vvx^*}^2)  \notag\\
		\leq{} & \tilde{V}_k - (1+\hat{\sigma}) \tilde{V}_{k+1} +  \bigg(\frac{\eta \gamma^2}{2} - \frac{c \hat{\lambda}_1 }{2} \bigg)\mathbb{E}(\norm{\vvx_{k+1} - \vvx_k }^2)  \notag \\
		&+ \bigg( \frac{1}{2 \eta} + \frac{c \hat{\sigma} \hat{\lambda}_n}{2} \bigg)\mathbb{E}(\norm{\vvx_{k+1} - \vvx^*}^2)  \notag \\
		& +  \bigg( c \beta^{-1} + 4 c \hat{\sigma} \bigg)\mathbb{E}(\norm{\vr_{k+1} - \vr^*}^2) \notag \\
		&	+ \frac{c \lambda_n (4\beta + 3)}{2}\mathbb{E}\bigg(\norm{\vve_{k+1}}^2) 
		+\norm{\vve_k}^2\bigg). \label{eq:vk-vk-1}
	\end{align}
	
	Next, we consider the term $\norm{\vr_{k+1} - \vr^*}^2$. Recalling \eqref{eq:grad-f}, we have
	\begin{align}
		& \hspace*{-5mm} \nabla f(\vvx_{k+1}) - \nabla f(\vvx^*) +c \mL_{\ts} (\vvx_{k+1} - \vvx_{k}) + \vdelta_k \notag \\
		%= {}&  c \mL_{\ts} (\vvx_k - \vvx_{k+1}) - \vphi_{k+1} + \vphi^* + c \mL (\vve_{k+1} - \vve_k) - \vdelta_k \notag \\
		= {}&   c \mL ( \vve_{k+1} - \vve_{k} ) - 2c \mM^{\tr}(\vr_{k+1} - \vr^*)  \label{eq:r-rstar}
	\end{align}
	Denote the left-side and right-side of equation \eqref{eq:r-rstar} as $\Xi_L$ and $\Xi_R$, respectively. 
	By applying $\|\vx + \vy\|^2  \geq \frac{1}{2} \|\vy\|^2 - \|\vx\|^2$ to $\norm{\Xi_R}^2$, we obtain 
	\begin{align}
		\norm{\Xi_R}^2 \geq{}&  \frac{1}{2} \norm{2c\mM^{\tr} (\vr_{k+1} - \vr^{\ast})}^2 - c^2 \norm{ \mL (\ve_{k+1}-\ve_{k})}^2  \notag \\
		\geq{}& 4 c^2 \lambda_2 \|\vr_{k+1} - \vr^{\ast}\|^2 - 2 c^2 \lambda_n^2(\norm{\ve_{k}}^2 + \norm{\ve_{k+1}}^2).  \notag
	\end{align}
	Regarding $\norm{\Xi_L}^2$, we have
	\begin{align}	
		\norm{\Xi_L}^2 \leq{}& 2 \ell^2 \|\vvx_{k+1} - \vvx^{\ast}\|^2 + 4 \Big(c^2 \hat{\lambda}_n^2 + \gamma^2 \Big) \|\vvx_{k+1} - \vvx_k\|^2.  \notag
	\end{align}
	Combining estimation on  $\norm{\Xi_R}^2$ and $\norm{\Xi_L}^2$ yields
	\begin{align}
		\mathbb{E}(\norm{\vr_{k+1} - \vr^*}^2)
		\leq {}& \frac{c^2 \hat{\lambda}_n^2 + \gamma^2}{c^2 \lambda_2} \mathbb{E}(\|\vvx_{k+1} - \vvx_k\|^2)  \notag \\
		{}&+ \frac{\lambda_n^2}{2 \lambda_2}\mathbb{E}((\norm{\vve_{k}}^2 + \norm{\vve_{k+1}}^2)) \notag \\
		{}&	+ \frac{\ell^2}{2 c^2 \lambda_2 }\mathbb{E}(\|\vvx_{k+1} - \vx^{\ast}\|^2). \label{eq:rk-rstar-bound}
	\end{align}	
	Substituting \eqref{eq:rk-rstar-bound} into \eqref{eq:vk-vk-1}, we have
	\begin{align}
		& \hspace*{-5mm} (1+ \hat{\sigma}) \tilde{V}_{k+1} -\tilde{V}_k \notag \\
		\leq{} &  \bigg( \frac{(  \beta^{-1} + 4  \hat{\sigma} ) (c^2 \hat{\lambda}_n^2 + \gamma^2)}{c \lambda_2} - \frac{c \hat{\lambda}_1 - \eta \gamma^2 }{2} \bigg)\mathbb{E}(\norm{\vvx_{k+1} - \vvx_k }^2) \notag \\
		& + \bigg( \frac{1}{2 \eta} + \frac{c \hat{\sigma} \hat{\lambda}_n}{2} + \frac{\ell^2 ( \beta^{-1} + 4  \hat{\sigma})}{2 c \lambda_2 } - v\bigg)\mathbb{E}(\norm{\vvx_{k+1} - \vvx^*}^2)  \notag \\
		&+\underbrace{\big(2\beta c\lambda_n+\frac{3c \lambda_n }{2} + \frac{(1 + 4 \hat{\sigma}\beta) c \lambda_n^2}{2 \lambda_2 \beta})}_{\Xi_1}\mathbb{E}(\|\vve_{k+1}\|^2) \notag \\ 
		&+ \underbrace{\bigg(\frac{3c \lambda_n }{2} + \frac{(1 + 4 \hat{\sigma}\beta) c \lambda_n^2}{2 \lambda_2 \beta}\bigg)}_{\Xi_2}\mathbb{E}(\norm{\vve_{k}}^2).  \label{eq:tilde-V_K-estimate}
	\end{align}
	According to the definition of $V_k$, we can know
	\begin{align}
		V_k-(1+\hat{\sigma})V_{k+1}=&\tilde{V}_k+r\mathbb{E}(\|\vve_{k}\|^2)-(1+\hat{\sigma})\tilde{V}_{k+1} \notag \\
		&-(1+\hat{\sigma})r\mathbb{E}(\|\vve_{k+1}\|^2). \label{eq:V_K-def}
	\end{align}
	In Lemma \ref{lemma:error-eve-comp}, the relationship between $\mathbb{E}(\|\vve_{k+1}\|^2)$ and $\mathbb{E}(\|\vve_{k}\|^2)$ is estimated, which can be seen in \eqref{eq:error-eve-comp}.  To estimate $V_k$, we substitute
	\eqref{eq:error-eve-comp} and \eqref{eq:tilde-V_K-estimate} into \eqref{eq:V_K-def}, thus obtaining
	\begin{align}
		& \hspace*{-8mm} (1+ \hat{\sigma}) {V}_{k+1} -{V}_k \notag \\
		\leq{}   \bigg(&  - \frac{c \hat{\lambda}_1 - \eta \gamma^2 }{2}+\frac{(  \beta^{-1} + 4  \hat{\sigma} ) (c^2 \hat{\lambda}_n^2 + \gamma^2)}{c \lambda_2}\notag\\
		&	+\frac{r(1+\hat{\sigma})\delta}{1-\sqrt{\delta}}+\Xi_1\frac{\delta}{1-\sqrt{\delta}} \bigg)\mathbb{E}(\norm{\vvx_{k+1} - \vvx_k }^2)  \notag \\
		&	+ \bigg( \frac{1}{2 \eta} + \frac{c \hat{\sigma} \hat{\lambda}_n}{2} + \frac{\ell^2 ( \beta^{-1} + 4  \hat{\sigma})}{2 c \lambda_2 } - v\bigg)\mathbb{E}(\norm{\vvx_{k+1} - \vvx^*}^2) \notag \\
		&  + \bigg(r(1+\hat{\sigma})\sqrt{\delta}+\Xi_1\sqrt{\delta}-r+\Xi_2 \bigg)\mathbb{E}(\norm{\vve_{k}}^2) \notag \\
		&	+ \bigg(r(1+\hat{\sigma})+\Xi_1\bigg)n\mu_{k+1}^2.	\label{eq:real-V_k-estiamte}
	\end{align}
	To obtain the result, the coefficients associated with the $\mathbb{E}(\|\vvx_{k+1}\|^2)$, $\mathbb{E}(\|\vvx_{k}\|^2)$ and $\mathbb{E}(\|\vve_{k+1}\|^2)$ in \eqref{eq:real-V_k-estiamte} are required 
	to be negative. That is
	\begin{align*}
		r-\Xi_2-\Xi_1\sqrt{\delta}-(1+\hat{\sigma})r\sqrt{\delta}&\geq 0{}, \\
		v - \frac{1}{2\eta} - \frac{\hat{\sigma} c \hat{\lambda}_n}{2} - \frac{\ell^2 (\beta^{-1}+ 4 \hat{\sigma} )}{2 c \lambda_2 } \geq{} 0,
	\end{align*}
	\begin{align*}
		\frac{c \hat{\lambda}_1-\eta\gamma^2}{2} 
		-\frac{(r+\Xi_1)\delta}{1-\sqrt{\delta}}-\frac{r\delta\hat{\sigma}}{1-\sqrt{\delta}} \\
		-\frac{ (c^2 \hat{\lambda}_n^2 + \gamma^2 ) (\beta^{-1}+ 4\hat{\sigma}  )}{c \lambda_2 } \geq{} 0.
	\end{align*}
	Define 
	\begin{align*}
		\tilde{\Xi}_1=\Xi_1|_{\hat{\sigma}=0}=\frac{3c\lambda_n}{2}+{2c\beta\lambda_n}+\frac{c\beta^{-1}\lambda_n^2}{2\lambda_2},\\
		\tilde{\Xi}_2=\Xi_2|_{\hat{\sigma}=0}=\frac{3c\lambda_n}{2}+\frac{c\beta^{-1}\lambda_n^2}{2\lambda_2}.
	\end{align*}
	As $\hat{\sigma}$ can be chosen as  a sufficiently small positive real number,
	it suffice to require
	\begin{align}
		r-\tilde{\Xi}_2-\tilde{\Xi}_1\sqrt{\delta}-r\sqrt{\delta}> 0{}, \label{eq:r-xishu-lower-bound}\\
		\frac{c \hat{\lambda}_1-\eta\gamma^2}{2} 
		-\frac{(r+\tilde{\Xi}_1)\delta}{1-\sqrt{\delta}}-\frac{ (c^2 \hat{\lambda}_n^2 + \gamma^2 )\beta^{-1}}{c \lambda_2 } >{} 0, \label{eq:w-range}\\
		v - \frac{1}{2\eta}  - \frac{\ell^2 \beta^{-1}}{2 c \lambda_2 } >{} 0\Rightarrow \eta >\frac{c\lambda_2}{2c\lambda_2v-\beta^{-1}\ell^2} .\notag
	\end{align}
	Regarding \eqref{eq:r-xishu-lower-bound}, we can get 
	$$ r>\frac{\tilde{\Xi}_2+\tilde{\Xi}_1\sqrt{\delta}}{1-\sqrt{\delta}}.$$ 
	To ensure \eqref{eq:w-range} can be satisfied, in \eqref{eq:w-range}, let $ r=\frac{\tilde{\Xi}_2+\tilde{\Xi}_1\sqrt{\delta}}{1-\sqrt{\delta}}$ and $\eta =\frac{c\lambda_2}{2c\lambda_2v-\beta^{-1}\ell^2}$, then we can obtain
	\begin{align} \label{eq:w-pre}
		\frac{c \hat{\lambda}_1}{2}-\frac{(\tilde{\Xi}_1+\tilde{\Xi}_2)\delta}{(1-\sqrt{\delta})^2}-\frac{c\lambda_2\gamma^2}{4c\lambda_2v-2\beta^{-1}\ell^2}
		-\frac{ (c^2 \hat{\lambda}_n^2 + \gamma^2 )}{c\beta \lambda_2 }>0
	\end{align}
	Rearrange the terms, then we can obtain:
	\begin{align} \label{eq:delta-range}
		\frac{\delta}{(1-\sqrt{\delta})^2}< \frac{	\frac{c \hat{\lambda}_1}{2}-\frac{c\lambda_2\gamma^2\beta}{4c\beta\lambda_2v-2\ell^2}-\frac{ (c^2 \hat{\lambda}_n^2 + \gamma^2 )}{c\beta\lambda_2 }}{3c\lambda_n+2c\beta\lambda_n+\frac{c\lambda_n^2}{\beta\lambda_2}}
	\end{align}
	Define the RHS of \eqref{eq:delta-range} as 
	$$F(c,\beta)=\frac{	\frac{\hat{\lambda}_1}{2}-\frac{\lambda_2\gamma^2\beta}{4c\beta\lambda_2v-2\ell^2}-\frac{ (c^2 \hat{\lambda}_n^2 + \gamma^2 )}{c^2\beta\lambda_2 }}{3\lambda_n+2\beta\lambda_n+\frac{\lambda_n^2}{\beta\lambda_2}}.$$
	It is easy to obtain 
	$$F(c,\beta)<F(\infty,\beta)=\frac{\frac{\hat{\lambda}_1}{2}-\frac{\hat{\lambda}_n^2}{\beta\lambda_2}}{3\lambda_n+2\beta\lambda_n+\frac{\lambda_n^2}{\beta\lambda_2}}.$$
	We find that $F(\infty,\beta)$ has a global maximum when $\beta=\beta^*=2u+\sqrt{4u^2+\frac{1}{2}\frac{\lambda_n}{\lambda_2}+3u}$, where $u=\frac{\hat{\lambda}_n^2}{\lambda_2\hat{\lambda}_1}$.
	So when $\delta$ is chosen such that $\frac{\delta}{(1-\sqrt{\delta})^2}<F(\infty,\beta^*)$, then there always exist a sufficient large $c$ which can ensure that \eqref{eq:delta-range} is satisfied.
\end{proof}
\section{Proof of Theorem 1 and Theorem 2}
%In Lemma \eqref{lemma:qk-estimate}, let $\mu_k=0$, then Theorem 1 can be proved.
%\subsection{Proof of Theorem 2 }
\begin{proof}
	According to Lemma \eqref{lemma:qk-estimate}, we can know
	\begin{align*}
		&	V_{k+1} \leq \frac{1}{1+\hat{\sigma}} V_k + n\psi \mu_{k+1}^2,\text{with}\notag \\
		&  \psi = r+\frac{1}{1+\hat{\sigma}}\left(1.5c\lambda_n + 2c{\beta}\lambda_n + \frac{(c\beta^{-1}+4\hat{\sigma} c)\lambda_n^2}{2\lambda_2}\right).
	\end{align*}
	When $\mu_k=0$, then we can obtain:
	\begin{align}
		V_{k+1} \leq \frac{1}{1+\hat{\sigma}} V_k \leq  \frac{1}{(1+\hat{\sigma})^2} V_{k-1}\dots\leq \frac{V_0}{(1+\hat{\sigma})^{k+1}}.
	\end{align}
	Define $\sigma=\frac{1}{1+\hat{\sigma}}$, then the proof of Theorem 1 is completed.
	
	When $\mu_k=\alpha \rho^{k-1}$,  we can get obtain
	\begin{align}
		V_{k+1} &\leq \frac{V_k}{1+\hat{\sigma}}  + n\psi \alpha \mu_{k+1}^{2}\\
		& \leq \frac{V_{k-1}}{(1+\hat{\sigma})^2}+\frac{n\psi\alpha\mu_k^2}{1+\hat{\sigma}}+n\psi\alpha\mu_{k+1}^2\leq\dots \notag \\
		&	\leq V_0\sigma^{k+1}+\sum_{t=0}^{k}{n\psi\alpha\rho^{2(k-t)}}\sigma^t \notag \\
		&{}=\sigma^{k+1}\big(V_0+n\psi\alpha\sigma^{-1}\sum_{t=0}^{k-1}(\frac{\rho^2}{\sigma})^{t}\big)  \label{eq:before-proof-the2}
	\end{align}
	Let $\tilde{\sigma}=\max(\sigma,\rho^2)$, according to \eqref{eq:before-proof-the2},  when $\tilde{\sigma}\neq \rho^2$, we can know
	\begin{align}
		V_{k+1}\leq & \tilde{\sigma}^{k+1}\big(V_0+n\psi\alpha\tilde{\sigma}^{-1})\sum_{t=0}^{k}(\frac{\rho^2}{\tilde{\sigma}})^{t}\big) \notag \\
		\leq & \tilde{\sigma}^{k+1}\big(V_0+n\psi\alpha\tilde{\sigma}^{-1})\frac{\tilde{\sigma}}{\tilde{\sigma}-\rho^2}.
	\end{align}
	When $\tilde{\sigma}=\rho^2$, we can get $V_{k+1}\leq k\tilde{\sigma}^{k+1}(V_0+n\psi\alpha\tilde{\sigma}^{-1})$. 
\end{proof}
\section{Proof of Corollary 1}
\begin{proof}
	Revisiting \eqref{eq:gradf-x}, since the compressor is unbiased, then we can obtain
	%\begin{align*}
	%	\Xi_3={}&c \mathbb{E}\big((\vvx_{k+1}-\vvx^{\ast})^{\tr} \mL (\vve_{k+1}-\vve_{k})\big)=\mathbf{0}, \\
	%	 \Xi_2 = {}& -2 c\mathbb{E}\big( \vvx_{k+1}^{\tr} \mM^{\tr} (\vr_{k+1} - \vr^*)\big) 
	%	 ={}-2c\mathbb{E} \big((\vvy_{k+1} - \vve_{k+1})^{\tr} \mM^{\tr} (\vr_{k+1} - \vr^*)\big)  \notag \\
	%	 ={}&  4 c \mathbb{E}\bigg(\norm{\vr_k - \vr^*}^2 - \norm{\vr_{k+1} - \vr^*}^2 - \norm{\vr_k - \vr_{k+1}}^2 \bigg).
	%\end{align*}
	\begin{align*}
		\Xi_3={}&c \mathbb{E}\big((\vvx_{k+1}-\vvx^{\ast})^{\tr} \mL (\vve_{k+1}-\vve_{k})\big)=-c\mathbb{E}(\vve_{k}^{\tr}\mL\vvx_{k+1}), \\ 
		={}& c\mathbb{E}\big((\vvy_{k+1}-\vve_{k+1})^{\tr}\mL\vve_{k}\big)=c\mathbb{E}\big((\vvy_{k+1})^{\tr}\mL\vve_{k}\big)\\
		={}& 2c\mathbb{E}\big((\vr_{k+1}-\vr_k)^{\tr}\mM\vve_{k}\big), \\ %(\color{red}{\vr_{k+1}-\vr_k}=\frac{1}{4}\mM\vvy_{k+1})
		\leq{}& 4c\mathbb{E}(\|\vr_{k+1}-\vr_k\|^2)+\frac{c\lambda_n}{2}\mathbb{E}(\|\vve_{k}\|^2) \\
		\Xi_2 = {}& -2 c\mathbb{E}\big( \vvx_{k+1}^{\tr} \mM^{\tr} (\vr_{k+1} - \vr^*)\big) \notag \\
		={}&-2c\mathbb{E} \big((\vvy_{k+1} - \vve_{k+1})^{\tr} \mM^{\tr} (\vr_{k+1} - \vr^*)\big)  \notag \\
		={}&4 c \mathbb{E}\bigg(\norm{\vr_k - \vr^*}^2 - \norm{\vr_{k+1} - \vr^*}^2 - \norm{\vr_k - \vr_{k+1}}^2 \bigg)\\
		{}&+c\mathbb{E}(\|\vve_{k+1}\|_{\mL}^2)\\
		\leq{}&  4 c \mathbb{E}\bigg(\norm{\vr_k - \vr^*}^2 - \norm{\vr_{k+1} - \vr^*}^2 - \norm{\vr_k - \vr_{k+1}}^2 \bigg)\\
		&+{}c\lambda_n\mathbb{E}(\|\vve_{k+1}\|^2).
	\end{align*}
	By reusing the strong convexity of $f$ as \eqref{eq:vk-vk-1}, we can know:
	\begin{align}
		& \hspace*{-5mm} v\mathbb{E}(\norm{\vvx_{k+1} - \vvx^*}^2)\notag \\
		\leq {}& \Xi_1 + \Xi_2 + \Xi_3 + \Xi_4    \notag  \\
		\leq {}& \tilde{V}_k - \tilde{V}_{k+1}  +  \bigg(\frac{\eta \gamma^2}{2} - \frac{c \hat{\lambda}_1 }{2} \bigg)\mathbb{E}(\norm{\vvx_{k+1} - \vvx_k }^2)    \notag \\ 
		{}&+\frac{1}{2 \eta}\mathbb{E}(\norm{\vvx_{k+1} - \vvx^*}^2) \label{eq:estimate-tilde-Vk}
	\end{align}
	
	According to the definition of $V_{k+1}$, we can obtain:
	\begin{align}
		& \hspace*{-3mm}	\tilde{V}_k-\tilde{V}_{k+1} \notag \\
		={}&V_k-(1+\sigma)V_{k+1}+\frac{c\sigma\hat{\lambda}_n}{2}\mathbb{E}(\|\vvx_{k+1}-\vvx^*\|^2)\notag \\
		{}&-r\mathbb{E}(\|\vve_{k}\|^2) +4c\sigma\mathbb{E}{\|\vr_{k+1}-\vr^*\|^2}
		+(1+\sigma)r\mathbb{E}(\|\vve_{k+1}\|^2)\notag \\
		\leq & V_k-(1+\sigma)V_{k+1}+\frac{c\sigma}{2}\hat{\lambda}_n\mathbb{E}(\|\vvx_{k+1}-\vvx^*\|^2) 
		\notag \\
		&+\frac{(1+\sigma)r\delta}{1-\sqrt{\delta}}\mathbb{E}(\|\vvx_{k+1}-\vvx_{k}\|^2)+r\big((1+\sigma)\sqrt{\delta}-1\big)\mathbb{E}(\|\vve_{k}\|^2) \notag \\
		&+4c\sigma\mathbb{E}{\|\vr_{k+1}-\vr^*\|^2} \notag \\
		\overset{\eqref{eq:rk-rstar-bound}}{\leq} & V_k-(1+\sigma)V_{k+1}+(\frac{c\sigma\hat{\lambda}_n}{2}+\frac{2\ell^2\sigma}{c\lambda_2})\mathbb{E}(\|\vvx_{k+1}-\vvx^*\|^2)\notag \\
		&+\left.\bigg(\frac{4c^2\sigma\hat{\lambda}_n^2}{c\lambda_2}+\frac{4\sigma\gamma^2}{c\lambda_2}+\frac{2\delta\lambda_n^2c\sigma}{(1-\sqrt{\delta})\lambda_2}
		+\frac{(1+\sigma)r\delta}{(1-\sqrt{\delta})} \right. \notag \\ 
		&\left.+\frac{\delta c \lambda_n}{1-\sqrt{\delta}}\right.\bigg)\mathbb{E}(\|\vvx_{k+1}-\vvx_k\|^2)	+\big(r(1+\sigma)\sqrt{\delta}-r\big)\mathbb{E}(\|\vve_{k}\|^2)\notag \\
		&+\frac{2\lambda_n^2c\sigma(1+\sqrt{\delta})}{\lambda_2}\mathbb{E}(\|\vve_{k}\|^2)+(\sqrt{\delta}c\lambda_n+\frac{c\lambda_n}{2})\mathbb{E}(\|\vve_{k}\|^2)
		\label{eq:pre-estimate-Vk}
		%	&\label{eq:pre-estimate-Vk}
	\end{align}
	Substitute \eqref{eq:pre-estimate-Vk} into \eqref{eq:estimate-tilde-Vk} and rearrange the terms, then we can get:
	\begin{align}
		&\hspace{-5mm}V_k-(1+\sigma)V_{k+1}\notag \\
		\geq&
		\bigg(\frac{c\hat{\lambda}_1}{2}-\frac{\eta\gamma^2}{2}-\frac{2\delta\lambda_n^2c\sigma}
		{(1-\sqrt{\delta})\lambda_2} -\frac{4c^2\sigma\hat{\lambda}_n^2+4\sigma\gamma^2}{c\lambda_2}\notag \\
		&\left.-\frac{(1+\sigma)r\delta}{1-\sqrt{\delta}}-\frac{\delta c \lambda_n}{1-\sqrt{\delta}}\bigg)\right.\mathbb{E}(\|\vvx_{k+1}-\vvx_k\|^2)\notag \\
		&+\bigg(r-r(1+\sigma)\sqrt{\delta}-\frac{2\lambda_n^2c\sigma(1+\sqrt{\delta})}{\lambda_2}\notag \\
		&-\sqrt{\delta}c\lambda_n-\frac{c\lambda_n}{2}\bigg)\mathbb{E}(\|\vve_{k}\|^2)  \notag\\ 
		&+\bigg(v-\frac{c\sigma\hat{\lambda}_n}{2}-\frac{2\ell^2\sigma}{c\lambda_2}-\frac{1}{2\eta}\bigg)\mathbb{E}(\|\vvx_{k+1}-\vvx^*\|^2) 
		\label{eq:V_k-linear}
	\end{align}
	To ensure the linear convergence of $V_k$, let the RHS of \eqref{eq:V_k-linear} be greater than $0$, that is, the coefficient of each term should always 
	be positive, which yields  
	\begin{align}
		&	v-\frac{c\sigma\hat{\lambda}_n}{2}-\frac{2\ell^2\sigma}{c\lambda_2}-\frac{1}{2\eta}\geq 0, \\[2mm]	
		&r-r(1+\sigma)\sqrt{\delta}-\frac{2\lambda_n^2c\sigma(1+\sqrt{\delta})}{\lambda_2}-\sqrt{\delta}c\lambda_n-\frac{c\lambda_n}{2}\geq 0 \notag \\[2mm]
		&	\frac{c\hat{\lambda}_1}{2}-\frac{\eta\gamma^2}{2}-\frac{2\delta\lambda_n^2c\sigma}
		{(1-\sqrt{\delta})\lambda_2}-\frac{4c^2\sigma\hat{\lambda}_n^2+4\sigma\gamma^2}{c\lambda_2},  \notag  \\
		&-\frac{(1+\sigma)r\delta}{1-\sqrt{\delta}}-\frac{\delta c \lambda_n}{1-\sqrt{\delta}}\geq 0 
		\label{eq:c-delta-range}
	\end{align}
	%Since $\sigma>0$ can be arbitrary small, then $r$ can also be arbitrary small. To satisfied  \eqref{eq:c-delta-range}, we just need
	Since $\sigma>0$ can be arbitrary small,  to satisfied  \eqref{eq:c-delta-range}, we  need
	\begin{align}\label{eq:unbiased-pre-c}
		&r(1-\sqrt{\delta})\geq \sqrt{\delta}c\lambda_n+\frac{c\lambda_n}{2}\Rightarrow r\geq\frac{\sqrt{\delta}c\lambda_n+\frac{c\lambda_n}{2}}{1-\sqrt{\delta}}, \notag \\
		&	v-\frac{1}{2\eta}\geq 0 \Rightarrow \eta \geq \frac{1}{2v}, \notag \\
		 & \frac{c\hat{\lambda_1}}{2}-\frac{\eta\gamma^2}{2}-\frac{r\delta}{1-\sqrt{\delta}}-\frac{\delta c \lambda_n}{1-\sqrt{\delta}}\geq 0.
	\end{align}
	%To ensure the existence of $\eta$ and due to $\gamma=2\ell$, 
	%it should be required that $c\geq\frac{2\ell^2}{v\hat{\lambda_1}}$.
	%\bibliographystyle{IEEEtran}
	%\bibliography{reference}
	Rearrange \eqref{eq:unbiased-pre-c}, we can get
	\begin{align}
		c\big(\frac{\hat{\lambda}_1}{2}-\frac{3\lambda_n}{2}\frac{\delta}{(1-\sqrt{\delta})^2}\big)>\frac{\gamma^2}{4v}.
	\end{align}
	Since $c>0$ and $\gamma=2\ell$, to make sure the existence of $c$, let  $\frac{\hat{\lambda}_1}{2}-\frac{3\lambda_n}{2}\frac{\delta}{(1-\sqrt{\delta})^2}>0$, then we can obtain
	\begin{align} \label{eq:unbiased-c-reqirement}
		c>\frac{\ell^2}{v}\frac{2(1-\sqrt{\delta})^2}{\hat{\lambda}_1(1-\sqrt{\delta})^2-3\lambda_n\delta}
	\end{align}
	When $\delta=0$, \eqref{eq:unbiased-c-reqirement} becomes
	$c>\frac{2\ell^2}{v\hat{\lambda}_1}$, 
	which is the requirement of the penalty parameter c in DQM \cite{TSP-Mokhtari-Shi-etc2016}. 
\end{proof}

\end{document}